\documentclass[10pt,reqno]{amsart} % please use amsart at 11pt

\usepackage{amssymb,latexsym}
\usepackage{cite} % to get Refs. [1,2,3] typeset as [1--3] automatically

\usepackage[height=190mm,width=130mm]{geometry} % this is the journal
\theoremstyle{definition}
\theoremstyle{plain}
\usepackage{amsfonts}
\usepackage{amsmath}
\usepackage{amscd}
\usepackage{amsthm}
\usepackage{bbm}
\usepackage{bm}
\usepackage{enumerate}
\usepackage{enumitem}
\usepackage{hyperref}
\usepackage{mathabx}
\usepackage{mathdots}
\usepackage{mathtools}

\usepackage{citehack}
\usepackage{longtable}
\usepackage[matrix,arrow,curve]{xy}

\usepackage{lipsum}

\allowdisplaybreaks

\newcommand{\C}{\mathbb{C}}

\newcommand{\F}{\mathcal{F}}
\newcommand{\wt}{\textup{wt}}
\theoremstyle{plain}

\newtheorem{lemma}{Lemma}

\newtheorem{proposition}{Proposition}

\theoremstyle{definition}

\theoremstyle{remark}
\newtheorem{remark}{Remark}

\newcommand{\Z}{\mathbb{Z}}
\newcommand{\V}{\mathcal V}
\numberwithin{equation}{section} % to get equations numbered
\newcommand{\N}{\ensuremath{\mathbb {N}}}

\newcommand{\U}{\ensuremath{\mathcal{U}}}

\newcommand{\g}{\ensuremath{\Gamma}}

\newcommand{\ps}{{\raise 1pt\hbox{\tiny (}}}

\newcommand{\pss}{{\raise 1pt\hbox{\tiny [}}}
\newcommand{\pdd}{{\raise 1pt\hbox{\tiny ]}}}
\newcommand{\pd}{{\raise 1pt\hbox{\tiny )}}}

\newcommand{\bs}{{\raise 1pt\hbox{\tiny [}}}
\newcommand{\bd}{{\raise 1pt\hbox{\tiny ]}}}

\def\cross{\mathinner{\mathrel{\raise0.8pt\hbox{$\scriptstyle>$}}
                 \joinrel\mathrel\triangleleft}}

\def\U{\mathcal{U}}
\def\W{\mathcal{V}}
\def\W{\mathcal{W}}
\def\K{\mathcal{K}}

\usepackage{stackrel}

\newcommand{\be}{\begin{equation}}
\newcommand{\ee}{\end{equation}}

\newcommand{\nn}{\nonumber \\}

\newcommand{\nc}{\newcommand}
\nc{\cali}{\mathcal}
\nc{\on}{\operatorname}
\nc{\Wick}{{\mb :}}

\nc{\ddz}{\frac{\partial}{\partial z}}
\nc{\ch}{\mbox{ch}}
\nc{\Oo}{{\cali O}}
\nc{\cond}{|\,}
\nc{\bib}{\bibitem}
\nc{\pone}{\Pro^1}
\nc{\pa}{\partial}
\nc{\arr}{\rightarrow}
\nc{\larr}{\longrightarrow}
\nc{\ket}{)}
\nc{\bra}{(}
\nc{\gam}{\bar{\gamma}}
\nc{\ep}{\epsilon}
\nc{\su}{\widehat{{\mf s}{\mf l}}_2}
\nc{\sw}{{\mf s}{\mf l}}
\nc{\h}{{\mf h}}
\nc{\n}{{\mf n}}
\nc{\ab}{\mf{a}}
\nc{\is}{{\mb i}}
\nc{\js}{{\mb j}}
\nc{\He}{{\cali H}}
\nc{\inv}{^{-1}}
\nc{\ol}{\overline}
\nc{\wh}{\widehat}
\nc{\dst}{\displaystyle}

\nc{\delt}{\partial_t}
\nc{\ddt}{\frac{\partial}{\partial t}}
\nc{\delx}{\partial_x}
\nc{\mb}{\mathbf}
\nc{\mf}{\mathfrak}

\nc{\mbb}{\mathbb}
\nc{\Ctt}{\C((t))}
\nc{\Ct}{\C[t,t\inv]}

\nc{\ghat}{\wh{\g}}

\nc{\un}{\underline}
\nc{\mc}{\mathcal}
\nc{\BB}{{\mc B}}
\nc{\bb}{{\mf b}}
\nc{\kk}{{\mf k}}
\nc{\frob}{\times}
\nc{\sm}{\setminus}
\nc{\Pp}{{\mathbb P}^1}
\nc{\Aa}{{\mc A}}

\nc{\AutO}{\on{Aut}\Oo}
\nc{\AUTO}{\un{\on{Aut}}\Oo}
\nc{\AUTK}{\un{\on{Aut}}\K}
\nc{\Heout}{\He_{\out}}
\nc{\Hetil}{{\widetilde\He}}
\nc{\wb}{\overline}

\nc{\Res}{\on{Res}}
\nc{\pitil}{\Pi}
\nc{\Ctil}{\wt{C}}
\nc{\auto}{\on{Aut} \Oo}
\nc{\phitil}{\wt{\phi}}
\nc{\gz}{\g_{\vec z}}
\nc{\tensorM}{\bigotimes_{i=1}^N{\mathbb M}_i}
\nc{\tensorW}{\bigotimes_{i=1}^N W_{\nu_i,k}}
\nc{\out}{\on{out}}

\nc{\m}{{\mathfrak m}}

\nc{\gx}{\g^0_{\vec x}}

\nc{\hx}{\He^0_{\vec x}}
\nc{\tensorpi}{\pi_{\nu_1,\ldots,\nu_N}^\kappa}
\nc{\Phizw}{\Phi_{\vec w}({\vec z})}
\nc{\Pro}{{\mathbb P}}

\nc{\De}{D}

\nc{\us}{\underset}
\nc{\Ll}{\mc L}
\nc{\dR}{\on{dR}}

\nc{\T}{{\mc T}}

\nc{\Xn}{\overset{\circ}X{}^n} \nc{\Dn}{\overset{\circ}D{}^n}
\nc{\Dxn}{\overset{\circ}D{}^n_x} \nc{\varphitil}{\wt{\varphi}}

\nc{\lf}{{\mf l}}
\nc{\Wir}{\on{Vir}}

\nc{\bfgn}{{\bf g}_n}
\nc{\bfzn}{{\bf z}_n}

\begin{document}
\title[Canonical torsor bundles of prescribed rational functions]
{Canonical torsor bundles of prescribed rational functions on complex curves}
%%
                                
%%%%%%%%%%%%%%%%%%%%%%%%%%%%%%%%%%%%%%%%%%%%%%%%%%%%%%%%%%%%%%%%%%%%%%%%%%%%%%%
\author{A. Zuevsky} 
\address{Institute of Mathematics \\ Czech Academy of Sciences\\ Zitna 25, Prague\\ Czech Republic}

\email{zuevsky@yahoo.com}

%%%%%%%%%%%%%%%%%%%%%%%%%%%%%%%%%%%%%%%%%%%%%%%%%%%%%%%%%%%%%%%%%%%%%%%%%
% You may repeat \author \address as often as necessary                 %
%%%%%%%%%%%%%%%%%%%%%%%%%%%%%%%%%%%%%%%%%%%%%%%%%%%%%%%%%%%%%%%%%%%%%%%%%
%%
\begin{abstract}
%%
%%%%%%%%%%%%%%%%%%%%%%%%%%%%%%%%%%%%%%%%%%%%%%%%%%%%%%%%%%%%%%%%%%%%%%%%%%5
%%
Prescribed rational functions constitute a subset of rational functions satisfying  
certain symmetry and analyticity conditions. 
%%%%
 Over a smooth complex curve $M$, 
we construct explicitly a bundle $\W_M$ with values in
prescribed rational functions.   
%%%%
%%%%%%%%%%%%%%%%%%%%%%%%%%%%%%%%%%%%%%%%%%%%%%%%%%%%%%%%%%%%%%%%%%%%%%%%%%%%%%%%%%%
%%
An intrinsic coordinate-independent formulation for such bundles  
 is given.
The construction presented in this paper is useful for studies of the  
 canonical cosimplicial cohomology of 
  infinite-dimensional Lie algebras on smooth manifolds, as well as for 
purposes of conformal field theory,  
 deformation theory, and the theory 
of foliations. 
%%%%

\bigskip 
AMS Classification: 53C12, 57R20, 17B69 
\end{abstract}

\keywords{Rational functions with prescribed properties, fiber bundles, complex manifolds}

\vskip12pt  % insert '\vskip12pt' while using '\twocolumn' command
%\vskip28pt % if there is no keywords

\maketitle
%%%
%%%%%%%%%%%%%%%%%%%%%%%%%%%%%%%%%%%%%%%%%%%%%%%%%%%%%%%%%%%%%%%%%%%%%%%%%%%%%%%%%%%%%%%%%%%%%%
\section{Data availability statement}
The author confirms  that: 

\medskip 
1.) All data generated or analysed during this study are included in this published article. 

\medskip 
2.)   Data sharing not applicable to this article as no datasets were generated or analysed during the current study.

%%%%%%%%%%%%%%%%%%%%%%%%%%%%%%%%%%%%%%%%%%%%%%%%%%%%%%%%%%%%%%%%%%%%%%%%%%%%%%%%%%%%%%%%%%
%%%%%%%%%%%%%%%%%%%%%%%%%%%%%%%%%%%%%%%%%%%%%%%%%%%%%%%%%%%%%%%%%%%%%%%%%%%%%%%%%%%%%%%%%%
\section{Introduction}
%%
%% 
%%%%%%%%%%%%%%%%%%%%%%%%%%%%%%%%%%%%%%%%%%%%%%%%%%%%%%%%%%%%%%%%%%%%%%%%%%%%%%%%%
%%%%%%%%%%%%%%%%%%%%%%%%%%%%%%%%%%%%%%%%%%%%%%%%%%%%%%%%%%%%%%%%%%%%%%%%%%%%%%%%% 
%%%
For purposes of continuous cohomology theory computation on complex manifolds   
\cite{Fei, Wag} and their foliations \cite{BS, Lo}, it is important to be able to define and construct 
objects such as bundles on auxiliary manifolds \cite{BS, PT, Wag}.
%%%% 
One is eager to introduce  
 such bundles so that their sections would be canonical, i.e., defined 
over abstract disks    
independent of the choice of local coorinates over marked points on a complex manifold.  
%%%% 
%%
In order to do this, one needs a precise description of transformation 
properties of sections under changes of coordinates.
%%%%
%%
%%%%%%%%%%%%%%%%%%%%%%%%%%%%%%%%%%%%%%%%%%%%%%%%%%%%%%%%%%%%%%%%%%%%%%%%%%%%%%%%%%%%%%%%%%%%%%%%%%%%%%%%%%
%%%%%%%%%%%%%%%%%%%%%%%%%%%%%%%%%%%%%%%%%%%%%%%%%%%%%%%%%%%%%%%%%%%%%%%%%%%%%%%%%%%%%%%%%%%%%%%%%%%%%%%%%
The construction of such bundles on a smooth complex curve $M$ is grounded on the notion of 
prescribed rational functions we introduce in this paper.  
This are functions with specified analytic behavior, 
expressed through matrix elements defined for infinite-dimensional Lie algebras,
 and satisfying certain symmetry properties. 
%%%%
%%%%%%%%%%%%%%%%%%%%%%%%%%%%%%%%%%%%%%%%%%%%%%%%%%%%%%%%%%%%%%%%%%%%%%%%%%%%%%%%%%%%%%%%%%%%%%%%

%%%%%%%%%%%%%%%%%%%%%%%%%%%%%%%%%%%%%%%%%%%%%%%%%%%%%%%%%%%%%%%%%%%%%%%%%%%%%%%%%%
%%%%%%%%%%%%%%%%%%%%%%%%%%%%%%%%%%%%%%%%%%%%%%%%%%%%%%%%%%%%%%%%%%%%%%%%%%%%%%%%%%
The purpose of this paper is to construct, 
 in an intrinsic coordinate-independent way,  
a canonical fiber bundle of rational 
functions with defined analytical properties (described in section \ref{rational})
 which we call prescribed rational functions. 
%%%%
%%%%%%%%%%%%%%%%%%%%%%%%%%%%%%%%%%%%%%%%%%%%%%%%%%%%%%%%%%%%%%%%%%%%%%%%%%%%%%
%%
The construction involves 
torsors and twists of an infinite-dimensional Lie algebra $\mathcal G$ by the group of 
automorphisms of local coordinates independent transformations of non-intersecting 
domains 
of a number of points on a smooth complex curve. 
%%%%
%%
Prescribed rational functions are obtained via non-degenerate bilinear pairings on the algebraic 
completion of the space of $\mathcal G$-valued formal series.  
%%%%
%%
One of the essential ideas of this paper is to use the language of prescribed rational functions from 
the very beginning avoiding non-commutative ingredients as much as possible.  
%%%%
Thus, analytical properties and convergence of such rational functions are assumed. 
%%%%

%%%%%%%%%%%%%%%%%%%%%%%%%%%%%%%%%%%%%%%%%%%%%%%%%%%%%%%%%%%%%%%%%%%%%%%%%%%%%%%%%%%%%%%%%%%%%%%%%%%%%%%%%%%%%%%%%%%%
%%%%%%%%%%%%%%%%%%%%%%%%%%%%%%%%%%%%%%%%%%%%%%%%%%%%%%%%%%%%%%%%%%%%%%%%%%%%%%%%%%%%%%%%%%%%%%%%%%%%%%%%%%%%%%%%%%%%
The idea to study of the cohomology of certain bundles on a smooth manifold $M$ and to
connect this to the cohomology of $M$ has first appeared in \cite{BS}. 
  Let $Vect(M)$ be the Lie algebra of vector fields on $M$.   
Bott and Segal 
proved that the Gelfand-Fuks cohomology $H^*(Vect(M))$ \cite{Fuks} 
 is isomorphic to the singular cohomology $H^*(E)$ 
of the space $E$ of continuous cross sections of a certain 
fiber bundle $\mathcal E$ over $M$. 
Authors of \cite{PT, Sm} continued to use advanced topological methods of \cite{BS} 
 for more general cosimplicial spaces of maps.   
%%%%
%%%%%%%%%%%%%%%%%%%%%%%%%%%%%%%%%%%%%%%%%%%%%%%%%%%%%%%%%%%%%%%%%%%%%%%%%%%%%%%%%%%

%%%%%%%%%%%%%%%%%%%%%%%%%%%%%%%%%%%%%%%%%%%%%%%%%%%%%%%%%%%%%%%%%%%%%%%%%%%%%%%%%%%
%%%%%%%%%%%%%%%%%%%%%%%%%%%%%%%%%%%%%%%%%%%%%%%%%%%%%%%%%%%%%%%%%%%%%%%%%%%%%%%%%%%
Let us introduce the general notations used in this paper. 
%%%%
We denote by boldface vectors of elements, e.g., ${\bf a}_n=(a_1, \ldots, a_n)$,  
 and the same for all types of objects used in the text.   
%%%%
We also express as $({\bf a}_j)_n$ the $j$-th component of ${\bf a}_n$. 
The prime $'$ denotes the ordinary derivative. 
%%%%
%%%%%%%%%%%%%%%%%%%%%%%%%%%%%%%%%%%%%%%%%%%%%%%%%%%%%%%%%%%%%%%%%%%%%%%%%%%%%%%%%%%%%%%%%%%%%%%%%%%%
%%
Let $M$ be a smooth complex curve, $\mathcal G$ be an infinite-dimensional Lie algebra, 
$G_{{\bf z}_n}$ be the graded (with respect to a grading operator $K_G$)  
algebraic completion of the space of formal series individually in each 
of ${\bf z}_n$-variables. 
%%%%
Assuming that there exists a non-degenerate bilinear pairing $(. , .)$ on $G_{{\bf z}_n}$,   
 we denote ${\bf x}_n=({\bf g}_n, {\bf z}_n \;{\bf dz}_n)$   
for ${\bf g}_n$ of the $n$-th power ${\bf G}_{{\bf z}_n}=G_{{\bf z}_n}^{\otimes n}$ of $G_{{\bf z}_n}$, and 
$G^*_{ {\bf z}_n }$ be the dual to $G_{{\bf z}_n}$ with respect to $(. , .)$.  
%%%%
%%%%%%%%%%%%%%%%%%%%%%%%%%%%%%%%%%%%%%%%%%%%%%%%%%%%%%%%%%%%%%%%%%%%%%%%%%%%%%%%
%%
In case when formal variables ${\bf z}_n$ are associated
to $n$ points ${\bf p}_n$ on $M$, 
we denote $G_{{\bf z}_n}$ by $G_{{\bf p}_n}$, and when ${\bf z}_n$
are substituted by local coordinates ${\bf t}_{{\bf p}_n}$ in vicinities 
of ${\bf p}_n$, we replace $G_{{\bf z}_n}$   
 by $G_{{\bf t}_{{\bf p}_n}}$.  
%%%%
%%%%%%%%%%%%%%%%%%%%%%%%%%%%%%%%%%%%%%%%%%%%%%%%%%%%%%%%%%%%%%%%%%%%%%%%%%%%%%%%%%%%%%%%%%%%
For fixed $\theta \in G^*_{{\bf p}_n}$,   
and varying ${\bf x}_n \in {\bf G}_{{\bf z}_n}$ we
 consider a vector $\overline{F}( {\bf x}_n)$ of matrix elements of the form 
\begin{equation}
\label{toppa}
F({\bf x}_n) = \left( \theta, f ({\bf x}_n) \right)  \in \C((z)),  
\end{equation}
where $F({\bf x}_n)$  depends implicitly on ${\bf g}_n \in {\bf G}_{{\bf z}_n}$.  
%%%%
%%%%%%%%%%%%%%%%%%%%%%%%%%%%%%%%%%%%%%%%%%%%%%%%%%%%%%%%%%%%%%%%%%%%%%%%%%%%%\
%%
We may view the vector $\overline{F}({\bf x}_n)$ of prescribed rational functions as a
section of a fiber bundle over a collection of non-intersecting punctured discs 
 ${\bf D}^\times_{ {\bf x}_n } = {\rm Spec}_{ {\bf x}_n}\; \C( ({\bf z}_j)_n )$, $1 \le j \le n$,  
 with an ${\rm End} \left( G_{ {\bf z}_n }\right)$-valued fiber $f({\bf x}_n) \in G_{{\bf z}_n}$. 
%%%%
%%%%%%%%%%%%%%%%%%%%%%%%%%%%%%%%%%%%%%%%%%%%%%%%%%%%%%%%%%%%%%%%%%%%%%%%%%%%%%%%
%%

%%%%%%%%%%%%%%%%%%%%%%%%%%%%%%%%%%%%%%%%%%%%%%%%%%%%%%%%%%%%%%%%%%%%%%%%%%%%%%%%%%%%%%%%%%%%%%%%%%%%
%%%%%%%%%%%%%%%%%%%%%%%%%%%%%%%%%%%%%%%%%%%%%%%%%%%%%%%%%%%%%%%%%%%%%%%%%%%%%%%%%%%%%%%%%%%%%%%%%%%%
%
In this paper we explain how to construct the bundle mentioned above 
 in the case when the space of prescribed rational functions 
 carries an action of the group ${\bf Aut}_{n} \; {\bm \Oo}^{(1)}_n$ of local 
 coordinates changes in vicinities of $n$ points on $M$.  
%%%%
%%%%%%%%%%%%%%%%%%%%%%%%%%%%%%%%%%%%%%%%%%%%%%%%%%%%%%%%%%%%%%%%%%%%%%%%%%%%%%%%%%%%%%%%%%%%%%%%%%%%
%%
 This means that the action of the group 
${\bf Aut}_n\; {\bm \Oo}^{(1)}_n 
={\rm Aut}_1\; {\Oo}^{(1)}_1\times \ldots \times {\rm Aut}_n\; {\Oo}^{(1)}_n$  
comes about by exponentiation 
of the action of corresponding Lie algebras $\left({\it Der}_{0, j} \; {\bm\Oo}^{(1)}_j \right)_n$, 
 $a \le j \le n$,  
via the action on $G_{ {\bf p}_n }$.  
%%%%
%%%%%%%%%%%%%%%%%%%%%%%%%%%%%%%%%%%%%%%%%%%%%%%%%%%%%%%%%%%%%%%%%%%%%%%%%%%%%%%%%%%%%%%%%%%%%%%%
%%%%%%%%%%%%%%%%%%%%%%%%%%%%%%%%%%%%%%%%%%%%%%%%%%%%%%%%%%%%%%%%%%%%%%%%%%%%%%%%%%%%%%%%%%%%%%%%%%%%%
%%
%%
The construction of ${\rm Aut} \; \Oo^{(1)}$ was a part of the formal geometry developed in 
 \cite{GK, GKF, BR}    
 as a method of applying
representation theory of infinite-dimensional Lie algebras to finite-dimensional 
geometry.
%%%%
%%
The action of ${\rm Aut}\; \Oo^{(1)}$  was also considered  
in \cite{H2}. 
%%%%
%%%%%%%%%%%%%%%%%%%%%%%%%%%%%%%%%%%%%%%%%%%%%%%%%%%%%%%%%%%%%%%%%%%%%%%%%%%%%%%%%%%%%%%%%%%%%%%%%%%%%%%
%%
A vector $\overline{F} ({\bf x}_n)$ of matrix elements of $F({\bf x}_n)$ gives rise to a section 
 $\overline{\F}( {\bf p}_n )$ of the intrinsic bundle $\W_M|_{ {\bf D}^\times_{ {\bf p}_n} }$   
with ${\rm End}\left( G_{ {\bf p}_n }\right)$-valued fibers. 
%%%

%%%%%%%%%%%%%%%%%%%%%%%%%%%%%%%%%%%%%%%%%%%%%%%%%%%%%%%%%%%%%%%%%%%%%%%%%%
%%%%%%%%%%%%%%%%%%%%%%%%%%%%%%%%%%%%%%%%%%%%%%%%%%%%%%%%%%%%%%%%%%%%%%%%%%
%%
 The representation in term of formal series in ${ {\bf t}_{ {\bf p}_n} }$ 
   allows us to find the precise transformation formula 
 for all elements of $G_{ {\bf p}_n }$ under the action of ${\bf Aut}_n\; {\bm \Oo}^{(1)}_n$.  
%%%%
%%%%  
%%%%%%%%%%%%%%%%%%%%%%%%%%%%%%%%%%%%%%%%%%%%%%%%%%%%%%%%%%%%%%%%%%%%%%%%%%%%%%%%%%%%%
We then use this formula to give an intrinsic geometric meaning to sections 
$\overline{\F}( {\bf p}_n)$ of the fiber bundle in coordinate-free formulation.   
%%%%
 Namely, we
attach to each $\mathcal G$, satisfying certain properties 
 a fiber bundle $\W_M$ on an arbitrary smooth complex 
curve $M$. 
%%%%
%%%%%%%%%%%%%%%%%%%%%%%%%%%%%%%%%%%%%%%%%%%%%%%%%%%%%%%%%%%%%%%%%%%%%%%%%%%%%%%%%%%%%%%%%%%%%%%%%%%%%%
%%
 Such geometric realization of $\overline{\F} ( {\bf p}_n )$ allows us to provide 
 a global geometric meaning to the space of prescribed rational functions $F({\bf x}_n)$    
 on arbitrary curves.  
%%%%
%%%%%%%%%%%%%%%%%%%%%%%%%%%%%%%%%%%%%%%%%%%%%%%%%%%%%%%%%%%%%%%%%%%%%%%%%%%%%%%%%%%%%%%%%%%%%%
%%%%%%%%%%%%%%%%%%%%%%%%%%%%%%%%%%%%%%%%%%%%%%%%%%%%%%%%%%%%%%%%%%%%%%%%%%%%%%%%%%%%%%%%%%%%%%%
Finally, we prove that the bundle $\W_M$ we constructed is canonical, i.e., its 
sections do not depend on a change ${\bf t}_{ {\bf p}_n} \mapsto \widetilde{\bf t}_{ {\bf p}_n}$ of coordinates  
around points ${\bf p}_n$ on $M$. 
%%%%
%%%%%%%%%%%%%%%%%%%%%%%%%%%%%%%%%%%%%%%%%%%%%%%%%%%%%%%%%%%%%%%%%%%%%%%%%%%%%%%%%%%%%%%%%%%%%%
%%%%%%%%%%%%%%%%%%%%%%%%%%%%%%%%%%%%%%%%%%%%%%%%%%%%%%%%%%%%%%%%%%%%%%%%%%%%%%%%%%%%%%%%%%%%%%
\section{Rational functions and differentials on abstract discs}
In this section we describe (partially following \cite{BZF}) the setup needed for formulation of further results.  
%%%%
Let $p$ be a point on a smooth complex curve $M$, and $t_p$ be a local coordinate in a vicinity of $p$.  
%%%%
 We replace the field of Laurent series $\C((t_p))$ 
by any complete topological algebra non-canonically isomorphic to 
$\C((t_p))$.
%%%%
%%%%%%%%%%%%%%%%%%%%%%%%%%%%%%%%%%%%%%%%%%%%%%%%%%%%%%%%%%%%%%%%%%%%%%%%%%%%%%%%%%%%%%%%%%%
%%%%%%%%%%%%%%%%%%%%%%%%%%%%%%%%%%%%%%%%%%%%%%%%%%%%%%%%%%%%%%%%%%%%%%%%%%%%%%%%%%%%%%%%%%%
\subsection{Abstract discs}
 We may consider the scheme underlying the $\C$-algebra $\C[[t_p]]$.   
Viewing $\C[[t_p]]$ as the ring of complex-valued functions on the affine scheme 
 $D_{t_p} = {\rm Spec}\;  \C[[t_p]]$, 
we call this scheme the standard disc $D_{t_p}$. 
%%%%
%%%%%%%%%%%%%%%%%%%%%%%%%%%%%%%%%%%%%%%%%%%%%%%%%%%%%%%%%%%%%%%%%%%%%%%%%%%%%%%%%%%%%%%
%%
As a topological space, $D_p$ can be described by the origin 
(corresponding to the maximal ideal $t_p \; \C[[t_p]]$) and 
 the generic point (the zero ideal).     
A morphism from $D$ to an affine scheme
$Z = {\rm Spec}\;  \mathcal R$, where $\mathcal R$ is a $\C$-algebra, 
is a homomorphism of algebras 
$\mathcal R \to \C[[t_p]]$.
 If $M$ is a curve, such a homomorphism may
be constructed by realizing $\C[[t_p]]$ as a completion of $\mathcal R$.
%%
%%%%%%%%%%%%%%%%%%%%%%%%%%%%%%%%%%%%%%%%%%%%%%%%%%%%%%%%%%%%%%%%%%%%%%%%%%%%%%%%%%%%%%
 Geometrically, this is an identification of the disc $D_p$ with the formal neighborhood of a point on the
curve $M$.   
%% 
%%%%
%%%%%%%%%%%%%%%%%%%%%%%%%%%%%%%%%%%%%%%%%%%%%%%%%%%%%%%%%%%%%%%%%%%%%%%%%%%%%%%%%%%%%%%%%%%%%%%%%%%%%%%%%%
%\begin{definition}
%%
An abstract disc is an affine scheme ${\rm Spec}\;  \mathcal R$, where $\mathcal R$ is a
$\C$-algebra isomorphic to $\C[[t_p]]$. 
%%    
%%
%\end{definition}
%%%%
%%%%%%%%%%%%%%%%%%%%%%%%%%%%%%%%%%%%%%%%%%%%%%%%%%%%%%%%%%%%%%%%%%%%%%%%%%%%%%%%%%%%
%%
For the abstract disk,   the maximal ideal $t_p \C[[t_p]]$ has a preferred generator $t_p$.
There is no preferred generator in the maximal ideal of $\mathcal R$, and no
preferred coordinate on an abstract disc. 
Denote 
by $\Oo_p$ the completion of the local ring of $M$. Then $\mathcal O_p$ is non-canonically isomorphic 
to $\Oo = \C[[t_p]]$. To specify such an isomorphism, or equivalently, an isomorphism
between
$D_p = {\rm Spec}\;  \Oo_p$, 
and $D_{t_p} = {\rm Spec}\;  \C[[t_p]]$, we need to choose a formal coordinate $t_p$ at $p \in M$, i.e., a topological
generator of the maximal ideal $\mathfrak m_p$ of $\Oo_p$.
 In general there is no preferred formal
coordinate at $p\in M$, and $D_p$ is an abstract disc.  
%%%%
%%%%%%%%%%%%%%%%%%%%%%%%%%%%%%%%%%%%%%%%%%%%%%%%%%%%%%%%%%%%%%%%%%%%%%%
%%%%%%%%%%%%%%%%%%%%%%%%%%%%%%%%%%%%%%%%%%%%%%%%%%%%%%%%%%%%%%%%%%%%%%%
\subsection{Rational functions attached to discs}
Now we would like to attach rational functions to  
%%%%
 the standard  $D_{t_p} = {\rm Spec}\;  \C[[t_p]]$ and to 
 any abstract discs $D_p$, where $p$ is a point on $M$.  
%%%%
%%%%%%%%%%%%%%%%%%%%%%%%%%%%%%%%%%%%%%%%%%%%%%%%%%%%%%%%%%%%%%%%%%%%%%%%%%%%%%%%%%%%%%%%%%%%
%%%%%%%%%%%%%%%%%%%%%%%%%%%%%%%%%%%%%%%%%%%%%%%%%%%%%%%%%%%%%%%%%%%%%%%%%%%%%%%%%%%%%%%%%%%%
%%
 We denote by $\mathcal K_x$  
  the field of fractions of the ring of integers  
$\mathbb Z$ is the rational field $\mathbb Q$, and the field of fractions of the polynomial ring 
$\mathcal K[{\bf t_p}_n]$ over a field $\mathcal K$ is the field of rational functions
 $\mathcal K({\bf t_p}_n)=\left\{(R_1({\bf t_p}_n))/(R_2({\bf t_p}_n)): 
 R_1, R_2 \in \mathcal K[{\bf t_p}_n] \right\}$. 
The field of fractions of an integral domain $\mathcal D$ is the smallest field containing $\mathcal D$, 
since it is obtained from $\mathcal D$ by adding the least needed to make $\mathcal D$ a field, 
namely the possibility of dividing by any non-zero element.
%%
%%%%
%%%%%%%%%%%%%%%%%%%%%%%%%%%%%%%%%%%%%%%%%%%%%%%%%%%%%%%%%%%%%%%%%%
%%
 If we choose a coordinate $t_p$ on
$D_p$, then we obtain isomorphisms $\mathcal O_p = \C[[t_p]]$ and $\mathcal K_p = \C((t_p))$. 
%%
%%%%%%%%%%%%%%%%%%%%%%%%%%%%%%%%%%%%%%%%%%%%%%%%%%%%%%%%%%%%%%%%%%%%%%%%%%%%%%%%
%%
 We denote by $D_p$ (resp., $D^\times_p$) the disc (resp., the punctured disc) 
at $p$, defined as ${\rm {\rm Spec}\; } \Oo_p^{(1)}$ (resp., ${\rm {\rm Spec}\; } \; {\mathcal K}_p$).
%%%%
%%%%%%%%%%%%%%%%%%%%%%%%%%%%%%%%%%%%%%%%%%%%%%%%%%%%%%%%%%%%%%%%%%%%%%%%%%%%%%%%%%%%%%%%%%
%%%%%%%%%%%%%%%%%%%%%%%%%%%%%%%%%%%%%%%%%%%%%%%%%%%%%%%%%%%%%%%%%%%%%%%%%%%%%%%%%%%%%%%%%%
\subsection{Differentials}
In this subsection we recall basic definitions concerning differentials \cite{BZF, Sch}.    
%%%%
%%%%%%%%%%%%%%%%%%%%%%%%%%%%%%%%%%%%%%%%%%%%%%%%%%%%%%%%%%%%%%%%%%%%%%%%%%%%%%%%%%%%%%%
%\begin{definition}
%%
 Let $k$ be a rational number. A $k$-differential on a smooth curve 
is by definition a section of the $k$-th tensor power of the canonical line bundle $\Omega$.  
%%
%%
%\end{definition}
%%%%
%%
Choosing a local coordinate $t_p$ we may trivialize $\Omega^{\otimes k}$ by the non-vanishing 
section $(dt_p)^{\otimes k}$. 
%%%%
%%%%%%%%%%%%%%%%%%%%%%%%%%%%%%%%%%%%%%%%%%%%%%%%%%%%%%%%%%%%%%%%%%%%%%%%%
%%
%%
Any section of $\Omega^{\otimes k}$ may then be written as $f(t_p)(dt_p)^{\otimes k}$.  
%%%%
%%
 If we choose another coordinate $\widetilde{t}_p = \rho(t_p)$, 
then the same section will be written as $g(\widetilde{t}_p)(d\widetilde{t}_p)^{\otimes k}$, 
where
\[
f(t_p) = g(\rho(t_p))(\rho'(t_p))^{\otimes k}.
\]
%%%%
%%%%%%%%%%%%%%%%%%%%%%%%%%%%%%%%%%%%%%%%%%%%%%%%%%%%%%%%%%%%%%%%%%%%%%
Now let us suppose that we have a section of $\Omega^{\otimes k}$ 
 whose representation by a
function does not depend on the choice of local coordinate, i.e., 
$g(\widetilde{t}_p) = f(\widetilde{t}_p)$, and
  $f(t_p) = f(\rho(t_p))(\rho'(t_p))^{\otimes k}$  
 for any change of variable $\rho(t_p)$.  
%%%%
When  we consider sections of $\Omega^{\otimes k}$ with values in a vector space that itself transforms
non-trivially under changes of coordinates,
 canonical sections may exist.
%%
%%%%%%%%%%%%%%%%%%%%%%%%%%%%%%%%%%%%%%%%%%%%%%%%%%%%%%%%%%%%%%%%%%%%%
%\begin{definition}
%%
  We call  
 $f(t_p)(dt_p)^{\otimes k}$ a canonical $k$-differential. 
%%
%\end{definition}
%%%%
%%
Let us denote by $\Omega_p$ the space of differentials on $D^\times_p$.  
%%%%
%%%%%%%%%%%%%%%%%%%%%%%%%%%%%%%%%%%%%%%%%%%%%%%%%%%%%%%%%%%%%%%%%%%%%%%
%%%%%%%%%%%%%%%%%%%%%%%%%%%%%%%%%%%%%%%%%%%%%%%%%%%%%%%%%%%%%%%%%%%%%%%
%%
In \cite{BZF} we find  the following lemma  
which we apply to $G_{t_p}$:  
%%
%%%%%%%%%%%%%%%%%%%%%%%%%%%%%%%%%%%%%%%%%%%%%%%%%%%%%%%%%%%%%%%%%%%%%
\begin{lemma}
 Given a linear map  
 $\rho : \mathcal K_p \to {\rm End} \left(  G_{t_p}  \right)$,
such that for any $\alpha \in G_{t_p}$ we have 
$\rho(\mathfrak m_p)^l \cdot \alpha = 0$,   
 for large enough $l$, where $\mathfrak m_p$ is the maximal ideal of $\Oo_p$ at $p$.  
Then
$\omega_p = \sum\limits_{n\in \Z}  
\rho(t^n_p)\; t_p^{-n-1} \; dt_p$,  
is a canonical ${\rm End} \left(  G_{t_p}  \right)$-valued differential on $D^\times_{t_p}$
  i.e., it is independent of the choice
of coordinate $t_p$. 
\end{lemma}
%%
%%%%%%%%%%%%%%%%%%%%%%%%%%%%%%%%%%%%%%%%%%%%%%%%%%%%%%%%%%%%%%%%%%%%%%%%%%%%%%%%%%%%%%%%%%%%%%%%%%
%%%%%%%%%%%%%%%%%%%%%%%%%%%%%%%%%%%%%%%%%%%%%%%%%%%%%%%%%%%%%%%%%%%%%%%%%%%%%%%%%%%%%%%%%%%%%%
\section{Rational functions with prescribed analytic behavior} 
\label{rational}
%% 
%%%%%%%%%%%%%%%%%%%%%%%%%%%%%%%%%%%%%%%%%%%%%%%%%%%%%%%%%%%%%%%%%%%%%%%%%%%%%%%%%%%%%%%%%
%%
In this section  the space of prescribed rational functions is defined as  rational  
functions with certain analytical and symmetric properties \cite{H2, H1}.   
Such rational functions depend implicitly on an infinite number of non-commutative parameters.  
%%
%%%%%%%%%%%%%%%%%%%%%%%%%%%%%%%%%%%%%%%%%%%%%%%%%%%%%%%%%%%%%%%%%%%%%%%%%%%%%%%%%%%%%%%%%%%%%%
%%%%%%%%%%%%%%%%%%%%%%%%%%%%%%%%%%%%%%%%%%%%%%%%%%%%%%%%%%%%%%%%%%%%%%%%%%%%%%%%%%%%%%%%%%%%%
\subsection{Rational functions originating from matrix elements}
\label{functional}
% 
%%%%%%%%%%%%%%%%%%%%%%%%%%%%%%%%%%%%%%%%%%%%%%%%%%%%%%%%%%%%%%%%%
%\begin{definition}
%%%%
Let $M$ be a complex manifold. Denote by  
 ${\bf p}_n$ be a set of $n$ points on $M$. We denote by ${\bf \U}_n$ a set of domains such that 
 ${\bf p}_n \in {\bf \U}_n$.  
Let ${\bf z}_{n}$ be  $n$ complex coordinates in $\U_n$ around origins ${\bf p}_n$. 
%%
%%
%%%%%%%%%%%%%%%%%%%%%%%%%%%%%%%%%%%%%%%%%%%%%%%%%%%%%%%%%%%%%%%%%%%%%%%%%%%%%%%%
In this paper we consider 
 meromorphic functions 
 of several complex variables  
 defined on sets of open  
 domains of $M$ with local coordinates ${\bf z}_{n}$    
 which are extendable to rational functions  
 on larger domains on $M$. We denote such extensions by $R(f({\bf z}_{n}))$.   
%%
%\end{definition}
%%%%
%%
%%%%%%%%%%%%%%%%%%%%%%%%%%%%%%%%%%%%%%%%%%%%%%%%%%%%%%%%%%%%%%%%%%%%%%%%%%%%%%%%%%%%%%%%
%%%%
%\begin{definition}
%%%%
Denote by $F_{n}\mathbb C$ the  
configuration space of $n \ge 1$ ordered coordinates in $\mathbb C^{n}$, 
 $F_{n}\mathbb C=\{{\bf z}_n \in \mathbb C^{n}\;|\; z_{i} \ne z_{j}, i\ne j\}$. 
%%
%%
%\end{definition}
%%%%% 
%% 
%%%%%%%%%%%%%%%%%%%%%%%%%%%%%%%%%%%%%%%%%%%%%%%%%%%%%%%%%%%%%
%%%%%%%%%%%%%%%%%%%%%%%%%%%%%%%%%%%%%%%%%%%%%%%%%%%%%%%%%%%%%
In order to work with objects having coordinate invariant formulation,   
for a set of ${\bf G}_{{\bf z}_n}$-elements ${\bf g}_{n}$  
we consider converging 
rational 
functions $f({\bf x}_{n})\in G_{{\bf z}_n}$ of ${\bf z}_{n} \in F_{n}\mathbb C$. 
%% 
%%
%%%%%%%%%%%%%%%%%%%%%%%%%%%%%%%%%%%%%%%%%%%%%%%%%%%%%%%%%%%%%%%%%%%%%%%%%%%%%%%%%%%%%%
%%%%%%%%%%%%%%%%%%%%%%%%%%%%%%%%%%%%%%%%%%%%%%%%%%%%%%%%%%%%%%%%%%%%%%%%%%%%%%%%%%%%%%
%%
%%%%%%%%%%%%%%%%%%%%%%%%%%%%%%%%%%%%%%%%%%%%%%%%%%%%%%%%%%%%%%%%%%%%%%%%%%%%%%%%%%%%%%
%%
%\begin{definition}
%%
For an arbitrary fixed $\theta \in G^*_{{\bf p}_n}$,    
we call 
 a map linear in ${\bf g}_n$ and ${\bf z}_n$, 
%%
%%%%%%%%%%%%%%%%%%%%%%%%%%%%%%%%%%%%%%%%%%%%%%%%%%%%%%%%%%%%%%%%%%%%%%%%%%%%%%%%%%%%%%
%%
\begin{equation} 
F: {\bf x}_n 
\mapsto   
\label{deff}
    R(\theta, f({\bf x}_n 
)), 
\end{equation}
  a rational function (we denote by $R$) in ${\bf z}_n$   
with the only possible poles at 
$z_{i}=z_{j}$, $i\ne j$.  
Abusing notations, we denote 
\[
F({\bf x}_n )= R\left(\theta, f({\bf x}_n)\right). 
\] 
%%
%\end{definition}
%%
%%
%%%%%%%%%%%%%%%%%%%%%%%%%%%%%%%%%%%%%%%%%%%%%%%%%%%%%%%%
%\begin{definition}
%%
%% 
We define  left action of the permutation group $S_{n}$ on $F({\bf z}_n)$ 
by
\[
\nc{\bfzq}{{\bf z}_n}
\sigma(F)({\bf x}_n)=F\left({\bf g}_n, {\bf z}_{\sigma(i)} \; {\bf dz}_{\sigma(i)}\right).  
\]
%%
%\end{definition}
%%
%%
%%%%%%%%%%%%%%%%%%%%%%%%%%%%%%%%%%%%%%%%%%%%%%%%%%%%%%%%%%%%%%%%%%%%%%%%%%%%%%%%%%%%%%%%%%%%%%%%%%
%%%%%%%%%%%%%%%%%%%%%%%%%%%%%%%%%%%%%%%%%%%%%%%%%%%%%%%%%%%%%%%%%%%%%%%%%%%%%%%%%%%%%%%%%%%%%%%%%%
\subsection{Conditions on rational functions} 
Let ${\bf z}_n \in F_{n}\C$.  
Denote by $T_G$ the translation operator \cite{H2}. 
We define now extra conditions on rational functions leading to the definition of restricted 
rational functions. 
%%
%%%%%%%%%%%%%%%%%%%%%%%%%%%%%%%%%%%%%%%%%%%%%%%%%%%%%%%%%%%%%%%%%%%%%%%%%%%%%%%%%%%%%%%%
%%%%%%%%%%%%%%%%%%%%%%%%%%%%%%%%%%%%%%%%%%%%%%%%%%%%%%%%%%%%%%%%%%%%%%%%%%%%%%%%%%%%%%%%
%\begin{definition}
%%
Denote by $(T_G)_i$ the operator acting on the $i$-th entry.
 We then define the action of partial derivatives on an element $F({\bf x}_n)$  
\begin{eqnarray}
\label{cond1}
\partial_{z_i} F({\bf x}_n)  &=& F((T_G)_i \; {\bf x}_n),   
\nn
\sum\limits_{i \ge 1} \partial_{z_i}  F({\bf x}_n)  
&=&  T_{G} F({\bf x}_n),   
\end{eqnarray}
%%
%\end{definition}
%%
and call it $T_{G}$-derivative property. 
%%
%%
%%%%%%%%%%%%%%%%%%%%%%%%%%%%%%%%%%%%%%%%%%%%%%%%%%%%%%%%%%%%%%%%%%%%%%
%%%%%%%%%%%%%%%%%%%%%%%%%%%%%%%%%%%%%%%%%%%%%%%%%%%%%%%%%%%%%%%%%%%%%%
%%
%\begin{definition} 
%%
%%
For   $z \in \C$,  let 
\begin{eqnarray}
\label{ldir1}
 e^{zT_G} F ({\bf x}_n)   
 = F({\bf g}_n, ({\bf z}_n +z)\; {\bf dz}_n ). 
\end{eqnarray}
%%
%%%%%%%%%%%%%%%%%%%%%%%%%%%%%%%%%%%%%%%%%%%%%%%%%%%%%%%%%%%%%%%%%%%%%%%%%%%%%%%%%%%%%%%%%%%%
 Let 
 ${\rm Ins}_i(A)$ denotes the operator of multiplication by $A \in \C$ at the $i$-th position. Then we define   
\begin{equation}
\label{expansion-fn}
F\left({\bf g}_n, {\rm Ins}_i(z) \; {\bf z}_n \; {\bf dz}_n\right)=  
F\left( {\rm Ins}_i (e^{zT_G}) \; {\bf x}_n\right),  
\end{equation}
are equel as power series expansions in $z$, in particular, 
 absolutely convergent
on the open disk $|z|<\min_{i\ne j}\{|z_{i}-z_{j}|\}$.  
%%
%\end{definition}
%%
%%
%%%%%%%%%%%%%%%%%%%%%%%%%%%%%%%%%%%%%%%%%%%%%%%%%%%%%%%%%%%%%%%%%%%%%%%%%%%%%%%
%%%%%%%%%%%%%%%%%%%%%%%%%%%%%%%%%%%%%%%%%%%%%%%%%%%%%%%%%%%%%%%%%%%%%%%%%%%%%%%%%
%%%%%%%%%%%%%%%%%%%%%%%%%%%%%%%%%%%%%%%%%%%%%%%%%%%%%%%%%%%%%%%%%%%%%%%%%%%%%%%
%%
%\begin{definition}
%%
A rational function has $K_G$-property   
if for $z\in \C^{\times}$ satisfies 
$(z\;{\bf  z}_n) \in F_{n}\C$,  
\begin{eqnarray}
\label{loconj}
z^{K_G } F ({\bf x}_n) =  
 F \left(z^{K_G} {\bf g}_n, 
 z\; {\bf z}_n \; {\bf dz}_n\right). 
\end{eqnarray}
%%
%\end{definition}
%% 
%%%%%%%%%%%%%%%%%%%%%%%%%%%%%%%%%%%%%%%%%%%%%%%%%%%%%%%%%%%%%%%%%%%%%%%%%%%%%%%%%%%%%%%%%
%%%%%%%%%%%%%%%%%%%%%%%%%%%%%%%%%%%%%%%%%%%%%%%%%%%%%%%%%%%%%%%%%%%%%%%%%%%%%%%%%%%%%%%%%
%%
\subsection{Rational functions with prescribed analytical behavior}
In this subsection we give the definition of rational functions with prescribed analytical behavior
on a domain of complex manifold $M$ of dimension $n$.    
We denote by $P_{k}: G \to G_{(k)}$, $k \in \C$,      
the projection of $G$ on $G_{(k)}$.
For each element $g_i \in G$, and $x_i=(g_i, z)$, $z\in \C$ let us associate a formal series  
$W_{g_i}(z)= W(x_i)=  
\sum\limits_{k \in \C }  g_{i} \; z^{k} \; dz $, $i \in \Z$. 
%%
%%%%%%%%%%%%%%%%%%%%%%%%%%%%%%%%%%%%%%%%%%%%%%%%%%%%%%%%%%%%%%%%%%%%%%%%%%%%%%
%%
  Following \cite{H1}, we formulate 
%%
%%%%%%%%%%%%%%%%%%%%%%%%%%%%%%%%%%%%%%%%%%%%%%%%%%%%%%%%%%%%%%%%%%%%%%%%
%%%%%%%%%%%%%%%%%%%%%%%%%%%%%%%%%%%%%%%%%%%%%%%%%%%%%%%%%%%%%%%%%%%%%%%%
%\begin{definition}
%%
\label{defcomp}
%%
%%%%%%%%%%%%%%%%%%%%%%%%%%%%%%%%%%%%%%%%%%%%%%%%%%%%%%%%%%%%%%%%%%%%%%%%%
We assume that there exist positive integers $\beta(g_{l', i}, g_{l", j})$ 
depending only on $g_{l', i}$, $g_{l'', j} \in G$ for 
$i$, $j=1, \dots, (l+k)n $, $k \ge 0$, $i\ne j$, $ 1 \le l', l'' \le n$.  
 Let   
${\bf l}_n$ be a partition of $(l+ k)n     
=\sum\limits_{i \ge 1} l_i$, and $k_i=l_{1}+\cdots +l_{i-1}$. 
For $\zeta_i \in \C$,  
define 
$h_i  
=F 
({\bf W}_{ { \bf g}_{ k_i+{\bf l}_i  }} ( 
 {\bf z}_{k_i + {\bf l}_i  }- \zeta_i ))$, 
%%
%\end{eqnarray}
%%
for $i=1, \dots, n$.
We then call a rational function $F$ satisfying properties \eqref{cond1}--\eqref{loconj}, 
a rational function with prescribed analytical behavior, if 
under the following conditions on domains, 
%%
%\[
%%
$|z_{k_i+p} -\zeta_{i}| 
+ |z_{k_j+q}-\zeta_{j}|< |\zeta_{i} -\zeta_{j}|$,   
%%
%\]
%%
for $i$, $j=1, \dots, k$, $i\ne j$, and for $p=1, 
\dots$,  $l_i$, $q=1$, $\dots$, $l_j$, 
the function   
%%
%%%%%%%%%%%%%%%%%%%%%%%%%%%%%%%%%%%%%%%%%%%%%%%%%%%%%%%%%%%%%%%%%%%%%%%5
%
%%
%%
 $\sum\limits_{ {\bf r}_n \in \Z^n}  
F( {\bf P_{r_{i}}  h_i}; (\zeta)_{l})$,     
is absolutely convergent to an analytically extension 
in ${\bf z}_{l+k}$, independently of complex parameters $(\zeta)_{l}$,
with the only possible poles on the diagonal of ${\bf z}_{l+k}$   
of order less than or equal to $\beta(g_{l',i}, g_{l'', j})$.   
 In addition to that, for ${\bf g}_{l+k}\in G$,  the series 
$\sum_{q\in \C}$  
$F( {\bf W(g_k}$, ${\bf P}_q ( {\bf W(g}_{l+k}, {\bf z}_k), {\bf z}_{ k + {\bf l} }))$,    
is absolutely convergent when $z_{i}\ne z_{j}$, $i\ne j$
$|z_{i}|>|z_{s}|>0$, for $i=1, \dots, k$ and 
$s=k+1, \dots, l+k$ and the sum can be analytically extended to a
rational function 
in ${\bf z}_{l+k}$ with the only possible poles at 
$z_{i}=z_{j}$ of orders less than or equal to 
$\beta(g_{l', i}, g_{l'', j})$. 
%%
%
%%
%\end{definition}
%%
%%%%%%%%%%%%%%%%%%%%%%%%%%%%%%%%%%%%%%%%%%%%%%%%%%%%%%%%%%%%%%%%%%%%%%%%%%%%%%%%%
%%%%
For $m \in \N$ and $1\le p \le m-1$, 
 let $J_{m; p}$ be the set of elements of 
$S_{m}$ which preserve the order of the first $p$ numbers and the order of the last 
$m-p$ numbers, that is,
$$J_{m, p}=\{\sigma\in S_{m}\;|\;\sigma(1)<\cdots <\sigma(p),\;
\sigma(p+1)<\cdots <\sigma(m)\}.$$
Let $J_{m; p}^{-1}=\{\sigma\;|\; \sigma\in J_{m; p}\}$.  
%%%
In addition to that, for some rational functions require the property: 
\begin{equation}
\label{shushu}
\sum_{\sigma\in J_{n; p}^{-1}}(-1)^{|\sigma|} 
\sigma( 
F ({\bf g}_{\sigma(i)}, {\bf z}_n) 
)=0. 
\end{equation}
%%
%%%%%%%%%%%%%%%%%%%%%%%%%%%%%%%%%%%%%%%%%%%%%%%%%%%%%%%%%%%%%%%%%%%%
%%%%%%%%%%%%%%%%%%%%%%%%%%%%%%%%%%%%%%%%%%%%%%%%%%%%%%%%%%%%%%%%%%%%
%% 
%%
Let us also introduce the following vector containing rational functions with properties  
described above. 
\begin{equation}
\label{setka}
\overline{F}({\bf x}_n)= \left[ 
F \left( {\bf g}_n, {\bf z}_n \; {\bf dz}_{{\it i}(n)} \right)  \right].    
\end{equation}
where ${\it i}(j)$, $j=1, \ldots, n$, are cycling permutations of $(1, \dots, n)$ starting with $j$. 
%%

%%
%%%%%%%%%%%%%%%%%%%%%%%%%%%%%%%%%%%%%%%%%%%%%%%%%%%%%%%%%%%%%%%%%%%%%
%%%%%%%%%%%%%%%%%%%%%%%%%%%%%%%%%%%%%%%%%%%%%%%%%%%%%%%%%%%%%%%%%%%%%
%%
Finally, we formulate 
%%
%\begin{definition}
%%
\label{poyma}
We define the space $\Theta \left(n, k, G_{ {\bf z}_n }, U\right)$ of prescribed rational functions 
as a space of vectors $\overline{F}_{n}({\bf x}_n)$ of the form \eqref{setka} of 
  complex $n$-variable restricted    
 rational functions 
 with prescribed analytical behavior 
on a $F_{n}\C$-domain $U \subset M$,  and   
satisfying 
 $T_G$- and $K_G$-properties \eqref{cond1}--%, \eqref{ldir1}, \eqref{expansion-fn}, 
\eqref{loconj},
  and \eqref{shushu}.  
%%
%\end{definition}
%%
%%%%%%%%%%%%%%%%%%%%%%%%%%%%%%%%%%%%%%%%%%%%%%%%%%%%%%%%%%%%%%%%%%%%%%%%%%%%%%%%%%%%%%%%%%%%%
%%%%%%%%%%%%%%%%%%%%%%%%%%%%%%%%%%%%%%%%%%%%%%%%%%%%%%%%%%%%%%%%%%%%%%%%%%%%%%%%%%%%%%%%%%%%%
\section{The bundle $\W_M$ of prescribed rational functions} 
%%%%%%%%%%%%%%%%%%%%%%%%%%%%%%%%%%%%%%%%%%%%%%%%%%%%%%%%%%%%%%%%%%%%%%%%%%%%%%%%%%%%%%%%%%%%%
%%
%%
\label{bundle}
%%
%%%%%%%%%%%%%%%%%%%
In this section we provide the construction of prescribed rational function bundle on $M$.   
%%%%
%%
%%%%%%%%%%%%%%%%%%%%%%%%%%%%%%%%%%%%%%%%%%%%%%%%%%%%%%%%%%%%%%%%%%%%%%%%%%%%%%%%%%%%
%%%%%%%%%%%%%%%%%%%%%%%%%%%%%%%%%%%%%%%%%%%%%%%%%%%%%%%%%%%%%%%%%%%%%%%%%%%%%%%%%%%%
%%
\subsection{Admissible Lie algebras} 
\label{algebra}
Let us assume that for an infinite-dimensional Lie algebra $\mathcal G$, 
the grading of $G_{t_p}$ is bounded from below by  
some subspace index $\kappa= \in \C$, i.e., 
 $G_{t_p} = \oplus_{i \ge {\rm Re} (\kappa) } G_{t_p, i}$, 
with finite dimensional $\dim G_{t_p, i} < \infty$ grading subspaces $G_{t_p, i}$.  
%%%%
%%
%%%%%%%%%%%%%%%%%%%%%%%%%%%%%%%%%%%%%%%%%%%%%%%%%%%%%%%%%%%%%%%%%%%%%%%%%%%%
%%%%%%%%%%%%%%%%%%%%%%%%%%%%%%%%%%%%%%%%%%%%%%%%%%%%%%%%%%%%%%%%%%%%%%%%%%%%
%% 
%%
 Then we have a filtration  
$G_{t_{p_j}, \le m} =\bigoplus\limits_{i \ge {\rm Re}(\kappa)}^m G_{t_{p_j}, i}$,    
of $G_{t_{p_j}}$ 
 by finite-dimensional ${\rm Aut}_{p_j}\; \Oo^{(1)}_j$-submodules, $1 \le j \le n$.   
%%
%%
%%%%%%%%%%%%%%%%%%%%%%%%%%%%%%%%%%%%%%%%%%%%%%%%%%%%%%%%%%%%%%%%%%%%%%%%%%%%%%%%%%%%%%%%%%%%%% 
%%
In addition to that, 
 we assume  that: $\mathcal G$  
  carries an action of ${\bf Der}_n \; {\bm \Oo}^{(1)}_n$;    
the element 
 $\left(- \partial_{t_p}\right)$      
acts as the translation operator 
  on $G_{t_p}$ semi-simply with integral
eigenvalues;  the Lie subalgebra $\left({\bf Der}_+\right)_n \; {\bm \Oo}^{(1)}_n$ acts locally nilpotently; 
and the operator $\left(- t_p \; \partial_{t_p}\right)$ provides a $\C$-grading.   
%%%%
%%%%%%%%%%%%%%%%%%%%%%%%%%%%%%%%%%%%%%%%%%%%%%%%%%%%%%%%%%%%%%%%%%%%%%%%%%%%%%%%%%%%%%%%%
%%%%%%%%%%%%%%%%%%%%%%%%%%%%%%%%%%%%%%%%%%%%%%%%%%%%%%%%%%%%%%%%%%%%%%%%%%%%%%%%%%%%%%%%%
%%
Finally, 
 we  assume also that 
the action of the Lie algebras ${\bf Der}_n\; {\bm \Oo}^{(1)}_n$ on $G_{t_{ {\bf p}_n} }$  
 can be exponentiated to an action of the
group ${\bf Aut}_n\; {\bm \Oo}^{(1)}_n$.  
%%%%
%%
%%%%%%%%%%%%%%%%%%%%%%%%%%%%%%%%%%%%%%%%%%%%%%%%%%%%%%%%%%%%%%%%%%%%%%%%%%%%%%%%%%%%%%%%%%%%
%%%%%%%%%%%%%%%%%%%%%%%%%%%%%%%%%%%%%%%%%%%%%%%%%%%%%%%%%%%%%%%%%%%%%%%%%%%%%%%%%%%%%%%%%%%%
%\begin{definition}
%%
\label{admissible} 
We call $\mathcal G$ subject to the assumptions of this subsection 
an admissible Lie algebra.   
%%
%\end{definition}
%%%%
%%%%%%%%%%%%%%%%%%%%%%%%%%%%%%%%%%%%%%%%%%%%%%%%%%%%%%%%%%%%%%%%%%%%%%%%%%%%%%%%%%%%%%%
%%%%%%%%%%%%%%%%%%%%%%%%%%%%%%%%%%%%%%%%%%%%%%%%%%%%%%%%%%%%%%%%%%%%%%%%%%%%%%%%%%%%%%%
%%
\subsection{Torsors and twists under groups of automorphisms}
We now explain how to collect elements of the space $\Theta\left(n, k, G_{ {\bf z}_n }, U   \right)$ of 
prescribed rational functions 
into an intrinsic object 
on a collection of abstract discs. 
%%%%
%%%%%%%%%%%%%%%%%%%%%%%%%%%%%%%%%%%%%%%%%%%%%%%%%%%%%%%%%%%%%%%%%%%%%%%%%%%%%%%%%%%%%%%%%%%%%%%%%%%
%%
%%
 We consider a configuration of $n$-points ${\bf p}_n$ on a complex curve $M$ lying in non-intersecting 
local disks, and we assume that at each point of ${\bf p}_n$ 
a coordinate changes independently of changing of coordinates 
on other disks. 
%%%%
%%%%
%%%%%%%%%%%%%%%%%%%%%%%%%%%%%%%%%%%%%%%%%%%%%%%%%%%%%%%%%%%%%%%%%%%%%%%%%%%%%%%%%%%
%%
%%
Therefore, the general element of the group of independent automorphisms of coordinates of $n$ points on $M$ 
${\bf Aut}_n\; {\bm \Oo}^{(1)}_{{\bf p}_n}$ 
has the form 
${\bf t}_{ {\bf p}_n } \mapsto \left({\bm \rho}_j ( {\bf t}_{{\bf p}_j})\right)_n$, $1 \le j \le n$.    
%%%%
%%
%%%%%%%%%%%%%%%%%%%%%%%%%%%%%%%%%%%%%%%%%%%%%%%%%%%%%%%%%%%%%%%%%%%%%%%%%%%%%%%%%%%%%%%%%%%%%
%%%%%%%%%%%%%%%%%%%%%%%%%%%%%%%%%%%%%%%%%%%%%%%%%%%%%%%%%%%%%%%%%%%%%%%%%%%%%%%%%%%%%%%%%%%%%
%%
 We recall here the notion of a torsor we  use to prove the independence 
of the choice of coordinates 
for prescribed rational functions. 
%%
%%%%%%%%%%%%%%%%%%%%%%%%%%%%%%%%%%%%%%%%%%%%%%%%%%%%%%%%%%%%%%%%%%%%%%%%%%%%%%%%%%%%%%%%%%%%%%%%%
%%
%\begin{definition}
%%
\label{torsor}
Let $\mathfrak G$ be a group, and $X$ a non-empty set. 
Then $X$ is called a $\mathfrak G$-torsor if it is equipped with a simply transitive right action of $\mathfrak G$,
i.e., given $\xi$, $\widetilde \xi \in X$, there exists a unique $h \in \mathfrak G$ such that 
$\xi \cdot h = \widetilde \xi$, 
where for $h$, $\widetilde{h} \in \mathfrak G$ the right action is given by 
$\xi \cdot (h \cdot \widetilde{h}) = (\xi \cdot  h) \cdot \widetilde{h}$. 
The choice of any $\xi \in X$ allows us to identify $X$ with $\mathfrak G$ by sending
 $\xi  \cdot h$ to $h$.
%%
%\end{definition}
%%%%
%%%%%%%%%%%%%%%%%%%%%%%%%%%%%%%%%%%%%%%%%%%%%%%%%%%%%%%%%%%%%%%%%%%%%%%%%%%%%%%%%%%%%%%%%%%%%%%%%%%%
Applying the definition of a group twist to the group  ${\rm Aut}\; \Oo^{(1)}$ and 
 its module $G_{z}$ we obtain 
%%%%
%%%%%%%%%%%%%%%%%%%%%%%%%%%%%%%%%%%%%%%%%%%%%%%%%%%%%%%%%%%%%%%%%%%%%%%%%%%%%%%%%%%
%%%%%%%%%%%%%%%%%%%%%%%%%%%%%%%%%%%%%%%%%%%%%%%%%%%%%%%%%%%%%%%%%%%%%%%%%%%%%%%%%%%
%\begin{definition}
%%
Given a ${\rm Aut}\; \Oo^{(1)}$-module $G_{z}$ and a ${\rm Aut}\; \Oo^{(1)}$-torsor $X$,  
one defines the $X$-twist of $G_z$ as the set 
\[
\V_X = G_z \; {{}_\times \atop      {}^{  {\rm Aut} \; \Oo^{(1)}  }     }X  
=  G_z \times  X/  \left\{ (g, a \cdot \xi) \sim  (ag, \xi) \right\}.
\]
for $\xi \in X$, $a \in {\rm Aut}\; \Oo^{(1)}$, and $g\in G_z$.   
%%
%\end{definition}
%%%%
%%%%%%%%%%%%%%%%%%%%%%%%%%%%%%%%%%%%%%%%%%%%%%%%%%%%%%%%%%%%%%%%%%%%%%%%%%%%%%%%5
%%
Given $\xi \in X$, we may identify $G_z$ with $\V_X$, 
by $g \mapsto (\xi, g)$.  
This identification depends
on the choice of $\xi$.
 Since ${\rm Aut}\; \Oo^{(1)}$ acts on $G_z$ by linear operators, the vector space 
structure induced by the above identification does not depend on the choice of $\xi$, 
and $\V_X$ is canonically a vector space.
%%%%
%%%%%%%%%%%%%%%%%%%%%
%%
If one thinks of $X$ as a principal ${\rm Aut}\; \Oo^{(1)}$-bundle 
over a point, then $\V_X$ is simply the associated vector bundle corresponding to $G_z$.  
 Any structure on $G_z$ (such as a bilinear pairing or multiplicative structure) that
is preserved by ${\rm Aut}\; \Oo^{(1)}$ will be inherited by $\V_X$.
%%

%%%%%%%%%%%%%%%%%%%%%%%%%%%%%%%%%%%%%%%%%%%%%%%%%%%%%%%%%%%%%%%%%%%%%%%%%%%%%%%%%%%%%%%%%%%%%%
%%%%%%%%%%%%%%%%%%%%%%%%%%%%%%%%%%%%%%%%%%%%%%%%%%%%%%%%%%%%%%%%%%%%%%%%%%%%%%%%%%%%%%%%%%%%%%
%%
%% 
Now we wish to attach to any disc a certain twist $\V_{t_p}$ 
of $G_{t_p}$, so that $G_{t_p}$ is attached to the standard disc, 
and for any coordinate $t_p$ on $D_p$
we have an isomorphism
\begin{equation}
\label{isomo}
i_{t_p, p} : G_{t_p} \; \widetilde{\to} \;  \V_{t_p}.
\end{equation}
%%%%
%%
We then associate to elements of $G_{t_p}$ sections of some bundles on $D^\times_{t_p}$. 
The system of isomorphisms $i_{t_p, p}$ should satisfy certain compatibility condition. 
%%%%
 Namely, if $t_p$ and 
$\widetilde{t}_p$ are two coordinates on $D_p$ such that $\widetilde{t}_p = \rho(t_p)$, then we obtain
an automorphism 
$(i^{-1}_{\widetilde{t}_p, p} \circ i_{t_p, p})$  
 of $G_{t_p}$.  
 The condition is that the assignment
$\rho (z) \mapsto  i^{-1}_{\widetilde{t}_p, p} \circ i_{t_p, p}$,  
defines a representation on $G_{t_p}$ of the group ${\rm Aut}\; \Oo^{(1)}$ of changes of coordinates.
 If this condition is satisfied, then $\V_{t_p}$ is canonically identified with the twist of $G_{t_p}$ by the 
 ${\rm Aut} \; \Oo^{(1)}$-torsor of formal coordinates at $p$.  
%%%% 

%%%%%%%%%%%%%%%%%%%%%%%%%%%%%%%%%%%%%%%%%%%%%%%%%%%%%%%%%%%%%%%%%%%%%%%%%%%%%%%%%%%%
%%
In the next  subsection we will show, given 
 the space $\Theta\left(n, k, G_{ {\bf z}_n }, U\right)$  
of prescribed rational functions associated to an admissible 
Lie algebra,   
one can attach to it 
 a vector bundle $\W_M$ on any smooth curve $M$. 
 I.e.,  the elements of $\Theta\left(n, k, G_{ {\bf z}_n }, U\right)$    
give rise to a collection of  coordinate-independent 
sections $\F({\bf p}_n)$ of the dual bundle $\W^*_M$ in the neighborhoods of a collection of points ${\bf p}_n \in M$.   
%%%%

%%%%%%%%%%%%%%%%%%%%%%%%%%%%%%%%%%%%%%%%%%%%%%%%%%%%%%%%%%%%%%%%%%%%%%%%%%%%%%%%%%%%%%%%%%%%%%%%%%%%%%%%%%%%%%%%%%%
%%%%%%%%%%%%%%%%%%%%%%%%%%%%%%%%%%%%%%%%%%%%%%%%%%%%%%%%%%%%%%%%%%%%%%%%%%%%%%%%%%%%%
 The construction is based on the principal bundle for the group 
${\bf Aut}_n\; {\bm \Oo}^{(1)}_n$,  
 which naturally exists on an arbitrary smooth curve and on any collection 
${\bf D}_{ {\bf p}_n }$ of non-intersecting discs.  
%%%%
%%%%%%%%%%%%%%%%%%%%%%%%%%%%%%%%%%%%%%%%%%%%%%%%%%%%%%%%%%%%%%%%%%%%%%%%%%%%%%%%%%%%%%%%%%%%%%%%%%%%%%%%%%%%%
%%
We denote by ${\it Aut}_{ {\bf p}_n }$  the set of all 
coordinates ${\bf t}_{{\bf p}_n}$ on disks ${\bf D}_{ {\bf p}_n }$,   
centered at points ${\bf p}_n$. 
%%%%
%%
It comes equipped with a natural right action of the group of automorphisms 
 ${\bf Aut}_n \; {\bm \Oo}^{(1)}_n$.    
 If $t_{p_i} \in {\it Aut}_{p_i}$, 
and $\rho(z_i) \in {\rm Aut}_i\; \Oo^{(1)}_i$, then $\rho_i(  t_{p_i} ) \in {\it Aut}_{p_i}$.  
%%%%
%%
 Furthermore, as it was shown in \cite{BZF} that 
$(\rho_i * \mu_i)( t_{p_i} ) = \mu_i(  \rho_i(t_{p_i})  )$,   
for $1 \le i \le n$, 
it  defines a right simply transitive action of ${\rm Aut}_i\; \Oo^{(1)}_i$  
 on ${\it Aut}_{p_i}$. 
%%%%%
%%  
%%%%%%%%%%%%%%%%%%%%%%%%%%%%%%%%%%%%%%%%%%%%%%%%%%%%%%%%%%%%%%%%%%%%%%%%%%%%%%%%
Next we ave 
%%
%%%%%%%%%%%%%%%%%%%%%%%%%%%%%%%%%%%%%%%%%%%%%%%%%%%%%%%%%%%%%%%%%%%%%%%%%%%%%%%%
%%
\begin{lemma}
The group ${\bf Aut}_n\; {\bm \Oo}^{(1)}_n$
 acts naturally on ${\it Aut}_{{\bf p}_n}$,  
 and is a ${\bf Aut}_n\; {\bm \Oo}^{(1)}_n$-torsor.  
\end{lemma}
%%%%
%%%%%%%%%%%%%%%%%%%%%%%%%%%%%%%%%%%%%%%%%%%%%%%%%%%%%%%%%%%%%%%%%%%%%%%%%%%%%%%%%%%
%%
Thus, we can define the following twist.
%%%%%%%%%%%%%%%%%%%%%%%%%%%%%%%%%%%%%%%%%%%%%%%%%%%%%%%%%%%%%%%%%%%%%%%%%%%%%%%%%%%%%%%%%%%%%%%%%
%%
%%
%\begin{definition}
%%
\label{twist}
We can introduce the ${\bf Aut}_n \; {\bm \Oo}^{(1)}_n$-twist of $G_{{\bf p}_n}$ 
\[
\V_{{\bf p}_n}= G_{{\bf p}_n}   \; 
{  {}_\times \atop      {}^{   {\bf Aut}_n \; {\bm \Oo}^{(1)}_n   }}
  \;{\it Aut}_{{\bf p}_n}.    
\]
%%%%
%%
 The original definition was given in 
\cite{BD, Wi}. 
%%%%
%%%%%%%%%%%%%%%%%%%%%%%%%%%%%%%%%%%%%%%%%%%%%%%%%%%%%%%%%%%%%%%%%%%%%%%%%%%%%%%%%%%%%%%%%%%%%%%%
%%
For each set of formal coordinates ${\bf t}_{{\bf p}_n}$ at ${\bf p}_n$, and ${\bf g}_n \in {\bf G}_{{\bf z}_n}$, 
 any element of the twist $\V_{ {\bf p}_n }$   
 may be written uniquely as a
pair 
 $\left({\bf g}_n, {\bf t}_{ {\bf p}_n } \right)$.  
%%
%%  
%\end{definition}
%%%%
%%%%%%%%%%%%%%%%%%%%%%%%%%%%%%%%%%%%%%%%%%%%%%%%%%%%%%%%%%%%%%%%%%%%%%%%%%%%%%%%%%%%%%%%%%%%%%%%%%%%%%
%%%%%%%%%%%%%%%%%%%%%%%%%%%%%%%%%%%%%%%%%%%%%%%%%%%%%%%%%%%%%%%%%%%%%%%%%%%%%%%%%%%%%%%%%%%%%%%%%%%%%%
\subsection{Definition of prescribed rational functions bundle}
Now let us formulate the definition of fiber bundle associated 
through vectors of elements 
$F({\bf x}_n)$ with ${\bf x}_n=({\bf g}_n, {\bf t}_{{\bf p}_n})$ 
to 
the space 
$\Theta \left(n, k, G_{ {\bf t}_{{\bf p}_n} } {\bf D}_{ {\bf t}_{ {\bf p}_n } }\right)$
of prescribed rational functions 
on any set of standard disks ${\bf D}_{ {\bf t}_{ {\bf p}_n } }$ 
 around points ${\bf p}_n$ 
with local coordinates ${\bf t}_{ {\bf p}_n }$. 
%%
%%%%%%%%%%%%%%%%%%%%%%%%%%%%%%%%%%%%%%%%%%%%%%%%%%%%%%%%%%%%%%%%%%%%%%
%%
For that purpose  
 we involve the notion of a principal bundle  
for the group ${\bf Aut}_{{\bf p}_n}\; {\bm \Oo}^{(1)}_n$  
%%%
 which naturally exists on an arbitrary smooth curve and on any disc.  
%%%%
%%%%%%%%%%%%%%%%%%%%%%%%%%%%%%%%%%%%%%%%%%%%%%%%%%%%%%%%%%%%%%%%%%%%%%%%%%%%%%%%%%%%%%%%%%%%%%%%%%%%

%%%%%%%%%%%%%%%%%%%%%%%%%%%%%%%%%%%%%%%%%%%%%%%%%%%%%%%%%%%%%%%%%%%%%%%%%%%%%%%%%%%%%%%%%%%%%%%%%%%%%%%%%%%%
%%
For the fiber space 
provided by vectors of elements $f({\bf x}_n) \in G_{  {\bf t}_{{\bf p}_n } }$,  
using the property of prescribed rational functions
we form a principal ${\bf Aut}_n \; {\bm \Oo}_n$-bundle,  which is 
  a fiber bundle $\W_M|_{ {\bf D}_{ {\bf t}_{ {\bf p}_n } }}$ defined by trivializations 
\[
i_{ {\bf t}_{ {\bf p}_n } }: \overline{F} ({\bf x}_n)=
 \left[   (\theta, f({\bf x}_n)  )  \right] \to {\bf D}_{ {\bf t}_{ {\bf p}_n } },  
\]
 together with a continuous 
 right action 
\[
 \overline F({\bf x}_n) \times {\bf Aut}_n \; {\bm \Oo}^{(1)}_n \to \overline F({\bf x}_n),  
\] 
such that ${\bf Aut}_n \; {\bm \Oo}^{(1)}_n$   
preserves  $\overline F({\bf x}_n)$,  
 i.e.,  $\zeta$, $\zeta.a$ are sections of $\W_M|_{ {\bf D}_{ {\bf t}_{ {\bf p}_n } }}$
 for all $a \in {\bf Aut}_n \; {\bm \Oo}^{(1)}_n$,  
and acts freely and transitively, i.e., the map $a \mapsto \zeta.a$ is a homeomorphism. 
%%%%
%%
%%%%%%%%%%%%%%%%%%%%%%%%%%%%%%%%%%%%%%%%%%%%%%%%%%%%%%%%%%%%%%%%%%%%%%%%%%%%%%%%%%%%%%%%%
Thus, we have \cite{BZF} 
%%
%%
%%%%%%%%%%%%%%%%%%%%%%%%%%%%%%%%%%%%%%%%%%%%%%%%%%%%%%%%%%%%%%%%%%%%%%%%%%%%%%%%%%%%%%
\begin{lemma}
 For $1 \le i \le n$, the projection ${\it Aut}_{{\bf p}_n} \rightarrow M$ is a principal 
${\bf Aut}_{{\bf p}_n}\; {\bm \Oo}_n$-bundle. 
The fiber of this bundle at points ${\bf p}_n$ is the  
${\bf Aut}_n \; {\bm \Oo}^{(1)}_n$-torsor ${\it Aut}_{{\bf p}_n}$.  
\end{lemma}
%%%%
%%%%%%%%%%%%%%%%%%%%%%%%%%%%%%%%%%%%%%%%%%%%%%%%%%%%%%%%%%%%%%%%%%%%%%%%%%%%%%%%%%%%%%%%%%%%%%%%%%%%%%%
%%%%%%%%%%%%%%%%%%%%%%%%%%%%%%%%%%%%%%%%%%%%%%%%%%%%%%%%%%%%%%%%%%%%%%%%%%%%%%%%%%%%%%%%%%%%%%%%%%%%%%%
As we observed above, we have a representation
 of the product of the group of autormorphisms
${\bf Aut}_n \; {\bm \Oo}^{(1)}_n$ 
on $G_{ {\bf t}_{ {\bf p}_n } }$.  
Then we obtain
%%
%%%%%%%%%%%%%%%%%%%%%%%%%%%%%%%%%%%%%%%%%%%%%%%%%%%%%%%%%%%%%%%%%%%%%%%%%%%%%
%%%%%%%%%%%%%%%%%%%%%%%%%%%%%%%%%%%%%%%%%%%%%%%%%%%%%%%%%%%%%%%%%%%%%%%%%%%%%
%\begin{definition}
%%
%%%%%%%%%%%%%%%%%%%%%%%%%%%%%%%%%%%%%%%%%%%%%%%%%%%%%%%%%%%%%%%%%%%%%%%%%%%%%5
%% 
%%
%%
Let $({\it Aut}_M)_n$ be the set of $n$-tuples of local coordinates ${\it Aut}_{ {\bf p}_n }$ 
all over the complex curve $M$.  
Given a finite-dimensional ${\bf Aut}_n \; {\bm \Oo}^{(1)}_n$-module $G_{i, {\bf p}_n}$,   
 let 
\[
\W_M|_{ {\bf D}_{ {\bf t}_{ {\bf p}_n } } }= G_{i, {\bf t}_{ {\bf p}_n } }  
{\times \atop \; {\bf Aut}_n \; {\bm \Oo}^{(1)}_n } \; ({\it Aut}_M)_n,    
\]
be the fiber bundle associated to $G_{ i, {\bf t}_{ {\bf p}_n } }$  and $({\it Aut}_M)_n$.  
%%%%
%%
%%%%%%%%%%%%%%%%%%%%%%%%%%%%%%%%%%%%%%%%%%%%%%%%%%%%%%%%%%%%%%%%%%%%%%%%%%%%%%%%%%%%%%%%%%%%%%%%%%%%
%%
Thus, $\W_M|_{ {\bf D}_{ {\bf t}_{ {\bf p}_n } }}$ is a bundle of finite rank over 
$M|_{ {\bf D}_{ {\bf t}_{ {\bf p}_n } }}$ whose fiber at a collection of points  
${\bf p}_n \in M$ is the vector $[f({\bf x}_n )]$.   
%%
%\end{definition}
%%%%
%%%%%%%%%%%%%%%%%%%%%%%%%%%%%%%%%%%%%%%%%%%%%%%%%%%%%%%%%%%%%%%%%%%%%%%%%%%%%%%%%%%%%%%%%%
%%%%%%%%%%%%%%%%%%%%%%%%%%%%%%%%%%%%%%%%%%%%%%%%%%%%%%%%%%%%%%%%%%%%%%%%%%%%%%%%%%%%%%%%%%
In a vicinity of every point of ${\bf p}_n$ on $M$ we can choose disks 
 ${\bf D}_{ {\bf p}_n }$  such that the bundle $\W_M$ over ${\bf D}_{ {\bf p}_n }$ is 
${\bf D}_{ {\bf p}_n } \times \F({\bf x}_n)$, where $\F({\bf x}_n)$ is a section of $\W_M$. 
%%%%
%%%%%%%%%%%%%%%%%%%%%%%%%%%%%%%%%%%%%%%%%%%%%%%%%%%%%%%%%%%%%%%%%%%%%%%%%%%%%%%%%%%%%%%%%%%
%%%%%%%%%%%%%%%%%%%%%%%%%%%%%%%%%%%%%%%%%%%%%%%%%%%%%%%%%%%%%%%%%%%%%%%%%%%%%%%%%%%%%%
%%
%%
The fiber bundle $\W_M$ 
with fiber $\left[f({\bf p}_n) \right]$ is a map 
$\W_M: \C^n  \rightarrow M$  where $\C^n$ is the total space 
of $\W_M$ and $M$ is its base space.  
%%%%
%%%%%%%%%%%%%%%%%%%%%%%%%%%%%%%%%%%%%%%%%%%%%%%%%%%%%%%%%%%%%%%%%%%%%%%%%%%%%%%%%%%%%%%%%%%%
%%%%%%%%%%%%%%%%%%%%%%%%%%%%%%%%%%%%%%%%%%%%%%%%%%%%%%%%%%%%%%%%%%%%%%%%%%%%%%%%%%%%%%%%%%%%%%
%% 
%%
 For every set of points ${\bf p}_n \in M$ with local disks  ${\bf D}_{ {\bf t}_{ {\bf p}_n }} $  
 $i_{ {\bf t}_{ {\bf p}_n }}^{-1}$ is homeomorphic to $ {\bf D}_{ {\bf t}_{ {\bf p}_n }}\times \C^n$. 
Namely, 
we have for 
$\left[ f( {\bf p}_n ) \right]:
 i_{   {\bf t}_{ {\bf p}_n }  }^{-1} \rightarrow  {\bf D}_{ {\bf t}_{ {\bf p}_n }} 
  \times \C^n$, 
that 
$\mathcal P \circ \left[  f (  {\bf p}_n  )    \right]    
= i_{    {\bf t}_{   {\bf p}_n  }    } 
|_{   i^{-1}_{     {\bf t}_{ {\bf p}_n }           } } 
  \left( {\bf D}_{    {\bf t}_{ {\bf p}_n }     }     ) \right) $,  
where $\mathcal P$ is the projection map on ${\bf D}_{  {\bf t}_{   {\bf p}_n   }  }$.   
%%
%%%%%%%%%%%%%%%%%%%%%%%%%%%%%%%%%%%%%%%%%%%%%%%%%%%%%%%%%%%%%%%%%%%%%%%%%%%%%%%%%%%%%%%%%%%%%%%%%%%
%%%%
Now we are able to formulate the definition of a prescribed rational functions bundle over a complex curve $M$.   
%% 
%%%%%%%%%%%%%%%%%%%%%%%%%%%%%%%%%%%%%%%%%%%%%%%%%%%%%%%%%%%%%%%%%%%%%%%%%%%%%%%%%%%%%%%%%%%%%%%%%%%%%%%%%%%%
%%%%%%%%%%%%%%%%%%%%%%%%%%%%%%%%%%%%%%%%%%%%%%%%%%%%%%%%%%%%%%%%%%%%%%%%%%%%%%%%%%%%%%%%%%%%%%%%%%%%%%%%%%%%
%\begin{definition}
%%
For an ${\bf Aut}_n\; {\bm \Oo}^{(1)}_n$-module $G_{ {\bf t}_{ {\bf p}_n } }$ 
which has a filtration by finite-dimensional submodules 
$G_{ {\bf t}_{ {\bf p}_n, i } }$, $i \ge 0$, we consider the directed 
inductive limit $\W_M$ of a system of finite rank bundles $\W^i_M$   
 on $M$ defined by embeddings 
 $\W^i_M \rightarrow 
\W^j_M$, 
for 
 $i \le j$, i.e., 
  $\W_M$ it as a fiber bundle of infinite
rank over $M$.  
%%%%
%%%%%%%%%%%%%%%%%%%%%%%%%%%%%%%%%%%%%%%%%%%%%%%%%%%%%%%%%%%%%%%%%%%%%5
%%
 Similarly, the dual bundle $\W^*_M$ is  
inverse system of bundles 
$(\W^i_M)^*$, $i \ge 0$, 
 and surjections 
$(\W^j_M)^* \rightarrow (\W^i_M)^*$,
for 
 $i \le  j$, 
  a projective limit of bundles of finite rank. 
%%
%\end{definition}
%%%%
%%%%%%%%%%%%%%%%%%%%%%%%%%%%%%%%%%%%%%%%%%%%%%%%%%%%%%%%%%%%%%%%%%%%%%%%%%%%%%%%
%%%%%%%%%%%%%%%%%%%%%%%%%%%%%%%%%%%%%%%%%%%%%%%%%%%%%%%%%%%%%%%%%%%%%%%%%%%%%%%%
%%
In the next subsection 
 we will identify sections $\F({\bf p}_n)$ of $\W_M|_ {{\bf D}^\times_{{\bf p}_n}}$   
with an operation which takes vector of elements of
 $F({\bf x}_n) \in \Theta\left(n, k, G_{ {\bf x}_n },  D^\times_{ {\bf  t}_{  {\bf p}_n} }  \right)$, 
and assigns to it a dual element.
%%%%
%%%%%%%%%%%%%%%%%%%%%%%%%%%%%%%%%%%%%%%%%%%%%%%%%%%%%%%%%%%%%%%%%%%%%%%%%%%%%%%%%%%%%%%%%%%%%%%%%%%
%%%%%%%%%%%%%%%%%%%%%%%%%%%%%%%%%%%%%%%%%%%%%%%%%%%%%%%%%%%%%%%%%%%%%%%%%%%%%%%%%%%%%%%%%%%%%%%%%%%
%%
\subsection{Explicit construction of canonical intrinsic setup for $\W_M$}
In order to be able to define a section $\overline{\F}(  {\bf p}_n  )$ 
 defined on abstract disks $D^\times_{  {\bf p}_n }$ 
in the coordinate independent 
description of the bundle $\W_M$, we have to associate in some way 
$\overline{\F}(  {\bf p}_n  )$ to $\overline{F}(  {\bf x}_n  )$. 
%%%%
%%
%%%%%%%%%%%%%%%%%%%%%%%%%%%%%%%%%%%%%%%%%%%%%%%%%%%%%%%%%%%%%%%%
Now let us give
%%%%%%%%%%%%%%%%%%%%%%%%%%%%%%%%%%%%%%%%%%%%%%%%%%%%%%%%%%%%%%%%%
%%
%%%%%%%%%%%%%%%%%%%%%%%%%%%%%%%%%%%%%%%%%%%%%%%%%%%%%%%%%%%%%%%%%%%%%%%%
%%%%%%%%%%%%%%%%%%%%%%%%%%%%%%%%%%%%%%%%%%%%%%%%%%%%%%%%%%%%%%%%%%%%%%%%
%\begin{definition}
%%
For each triple 
 ${\bf p}_n$,  ${\bf g}_n \in {\bf G}_{{\bf z}_n}$, and  
fixed $\theta \in G^*_{{\bf p}_n}$,   
 we define a intrinsic $\C^n$-valued meromorphic   
section $\overline{ \F}({\bf p}_n)$ 
of the bundle $\W^*_{ {\bf  D}_{ {\bf p}_n} }$ on the punctured discs 
${\bf D}^\times_{\bf p_n}$ by  
an operation
%%
%%%%%%%%%%%%%%%%%%%%%
\begin{equation}
\label{mordo}
 \left(\theta,  \;  {\bf g}_n, \; {\bf p}_n \right)  
 \mapsto \left( \theta, \F_{  i_{ {\bf p}_n } }\right),  
\end{equation}
%%
%%%%%%%%%%%%%%%%%%%%%%%%%%%%%%%%%%%%%%%%%%%%%%%%%%%%%%%%%%%%%%%%%%%%%%%%%%%%%%%%%%%%%%%%%%%%%%%%%%%%%%%%%%%%%%%
%%
assigning to a vector $\overline{\F}({\bf p}_n)$ of $\W_{ {\bf D}_{ {\bf p}_n   } }$  
an element 
of ${\bm{ \mathcal K}}_{  {\bf p}_n  }$
 (i.e., functions on ${\bf D}^\times_{  {\bf p}_n }$),
 defined by the $\W^*_{  {\bf D}_{  {\bf p}_n }}$-fiber
 $\F_{i_{ t_{ {\bf p}_n } } }\in G_{ {\bf p}_n }$. 
%%
%%
%\end{definition}
%%%%
%%%%%%%%%%%%%%%%%%%%%%%%%%%%%%%%%%%%%%%%%%%%%%%%%%%%%%%%%%%%%%%%%%%%%%%%%%%%%%%%%%%%%%%%%%%%%%%%%%%%%%%
%%%%%%%%%%%%%%%%%%%%%%%%%%%%%%%%%%%%%%%%%%%%%%%%%%%%%%%%%%%%%%%%%%%%%%%%%%%%%%%%%%%%%
%%
 We now formulate  
the main statement of this paper which is contained in the following proposition for   
 prescribed rational functions 
  bundle $\W_M$.   
%%
%%%%%%%%%%%%%%%%%%%%%%%%%%%%%%%%%%%%%%%%%%%%%%%%%%%%%%%%%%%%%%%%%%%%%%%%%%%%%%%
%%%%%%%%%%%%%%%%%%%%%%%%%%%%%%%%%%%%%%%%%%%%%%%%%%%%%%%%%%%%%%%%%%%%%%%%%%%%%
\begin{proposition}
\label{mainpro}
A $\C^n$-valued canonical (i.e., independent
of the choice of coordinates ${\bf t}_{  {\bf p}_n }$ on ${\bf D}^\times_{{\bf p}_n}$) 
section $\overline \F({\bf p}_n)$  of the bundle $\W^*_M|_{ {\bf D}^\times_{{\bf p}_n} }$ on  
of the $G_{ {\bf p}_n }$-valued fibers  
$\F_{ i_{ {\bf t}_{{\bf p}_n} }}$  defined by \eqref{isomo} on 
${\bf D}^\times_{{\bf p}_n}$ dual to $\W_M|_{D^\times_{ {\bf p}_n} }$ is given by the formula
%%%%
%%%%%%%%%%%%%%%%%%%%%%%%%%%%%%%%%%%%%%%%%%%%%%%%%%%%%%%%%%%%%%%%%%%%%%%%%%%%%%%%%%%%%%%%%
\begin{equation}
\label{mainfo}
\overline{\F}({\bf p}_n)= 
\left[ \left( {\bf y}_n,  \F_{ i_{ {\bf t}_{  {\bf p}_n  } } }({\bf h}_n))  
   \right) \right]   
= \left[ \left( \theta, f({\bf x}_n)  \right) \right]= \overline{F}({\bf x}_n), 
\end{equation}
%%
%%%%%%%%%%%%%%%%%%%%%%%%%%%%%%%%%%%%%%%%%%%%%%%%%%%%%%%%%%%%%%%%%%%%%%%%%%%%%%%%%%%%%%%%%%%%%%%%%%%%%%%%
%%
for ${\bf y}_n=(\theta, {\bf t}_{ {\bf p}_n} )$, 
$f({\bf x}_n) \in G_{ {\bf t}_{ { \bf p}_n }}$,   
 where ${\bf t}_{{\bf p}_n}$ are coordinates on ${\bf D}^\times_{{\bf p}_n}$, and ${\bf h}_n \in {\bf G}_{{\bf z}_n}$. 
\end{proposition}
%%%%
%%%%%%%%%%%%%%%%%%%%%%%%%%%%%%%%%%%%%%%%%%%%%%%%%%%%%%%%%%%%%%%%%%%%%%%%%%%%%%%%%%%%%%%
%%%%%%%%%%%%%%%%%%%%%%%%%%%%%%%%%%%%%%%%%%%%%%%%%%%%%%%%%%%%%%%%%%%%%%%%%%%%%%%%%%%%%%%
%%
\begin{proof}
%%
%%%%%%%%%%%%%%%%%%%%%%%%%%%%%%%%%%%%%%%%%%%%%%%%%%%%%%%%%%%%%%%%%%%%%%%%%%%%%%
%%%%%%%%%%%%%%%%%%%%%%%%%%%%%%%%%%%%%%%%%%%%%%%%%%%%%%%%%%%%%%%%%%%%%%%%%%%%%%
%% 
%%
Now let us proceed with the explicit construction of $\F_{ i_{  { {\bf p}_n }  }}$.  
%%
%%%%
%%%%%%%%%%%%%%%%%%%%%%%%%%%%%%%%%%%%%%%%%%%%%%%%%%%%%%%%%%%%%%%%%%%%%%%%%%%%%%%
%%
By choosing coordinates ${\bf t}_{{\bf p}_n}$ on a collection of discs ${\bf D}^\times_{{\bf p}_n}$,  
 we obtain a trivialization  
%%
%%%%%%%%%%%%%%%%%%%%%%%%%%%%%%%%%%%%%%%%%%%%%%%%%%%%%%%%%%%%%%%%%%%%%%%%%%%%
%%
%%
\[
i_{ {\bf t}_{ {\bf p}_n} }: \overline F \left( G [[{\bf t}_{{{\bf p}_n}}]] \right) 
 \; {\widetilde{} \atop \rightarrow} \;  \Gamma \left( \W_M|_{ {\bf D}^\times_{ {\bf t}_{{\bf p}_n} } } \right),
\]   
 of the bundle $\W_{ {\bf  D}^\times_{ {\bf p}_n} }$
which we call the ${\bf t}_{{\bf p}_n}$-trivialization. 
%%%%
%%%%%%%%%%%%%%%%%%%%%%%%%%%%%%%%%%%%%%%%%%%%%%%%%%%%%%%%%%%%%%%%%%%%%%%%%
%%
%%%%%%%%%%%%%%%%%%%%%%%%%%%%%%%%%%%%%%%%%%%%%%%%%%%%%%%%%%%%%%%%%%%
We also obtain trivializations of the fiber 
 $G_{ {\bf p}_n } \; {\widetilde{}  \atop \rightarrow} 
\; \gamma \left( \W_M|_{ { \bf D}^\times_{ {\bf t}_{p_n}} }  \right)$, 
 and its dual
$G^*_{ {\bf p}_n }  \; {\widetilde{}  \atop \rightarrow} \;  
\gamma\left(\W^*_M|_{ { \bf D}^\times_{ {\bf t}_{p_n}} }    \right)$. 
%%
%%
%%%%%%%%%%%%%%%%%%%%%%%%%%%%%%%%%%%%%%%%%%%%%%%%%%%%%%%%%%%%%%%%%%%%%%%%%%%%%%%%%%%%%%%%%%
%%
Let us denote by $\left( {\bf g}_n, {\bf t}_{ {\bf p}_n } \right)$  the image  
of ${\bf g}_n \in {\bf G}_{{\bf z}_n}$ in $\W_M|_{ { \bf D}^\times_{ {\bf t}_{p_n}}} $  
and by $\left( \theta,  {\bf t}_{{\bf p}_n} \right)$ 
the image of $\theta \in G^*_{ {\bf p}_n  }$ in $\W^*_M|_{  { \bf D}^\times_{   {\bf t}_{p_n}     }} $    
under ${\bf t}_{{\bf p}_n}$-trivialization. 
%%%%
%%
%%%%%%%%%%%%%%%%%%%%%%%%%%%%%%%%%%%%%%%%%%%%%%%%%%%%%%%%%%%%%%%%%%%%%%%%%%%%%%%%%%%%%%%%
%%
%%
In order to define the required section $\overline{\F}(  {\bf p}_n )$ 
with respect to these trivializations through its matrix
elements   
we need to attach an element of $(\C( {\bf t}_{p_n} ))^n$ 
 to each triple $\left( {\bf g}_n, {\bf t}_{ {\bf p}_n }\right) 
\in \W_M|_{ { \bf D}^\times_{ {\bf t}_{p_n} }}$,  
$(\theta,  {\bf t}_{{\bf p}_n} ) \in \W^*_M|_{ { \bf D}^\times_{ {\bf t}_{p_n}}}$,  
and a section $i_{ {\bf t}_{{\bf p}_n}}({\bf h}_n)$ of $\W|_{ {\bf D}^\times_{ {\bf t}_{ {\bf p}_n }} }$ 
for ${\bf h}_n \in F\left(G_{  {\bf t}_{{\bf p}_n} }\right)$. 
%
%%%%
%%%%%%%%%%%%%%%%%%%%%%%%%%%%%%%%%%%%%%%%%%%%%%%%%%%%%%%%%%%%%%%%%%%%%%%%
%%
 The operation we define above is $\C^n$-linear in ${\bf g_n}$,   
 and $\theta \in G^*_{  {\bf p}_n }$ and $\C[[ (  {\bf t}_{  {p_i} }   )_n  ]]$-linear in $F( {\bf x}_n )$.  
%%%%
%%%%%%%%%%%%%%%%%%%%%%%%%%%%%%%%%%%%%%%%%%%%%%%%%%%%%%%%%%%%%%%%%%%%%%%%%%%%%
 It is sufficient to assign a function to the triples 
${\bf x}_n$, $\theta \in G_{ {\bf p}_n }$, ${\bf h}_n \in {\bf G}_{{\bf z}_n}$ 
in the ${\bf t}_{{\bf p}_n}$-trivialization. 
%%%%
Thus, we identify a $\C^n$-valued  
section $\overline {\widetilde{\F} } ({\bf p}_n)$ of  
$\W^*_{ {\bf D}^\times_{{\bf p}_n} }$, 
with the section $\overline F( {\bf x}_n)$ of $\W_{ {\bf D}^\times_{{\bf p}_n} }$ by means 
of formula \eqref{mainfo}. 
%%
%%%%%%%%%%%%%%%%%%%%%%%%%%%%%%%%%%%%%%%%%%%%%%%%%%%%%%%%%%%%%%%%%%%%%%%%%%%%%%%%%%%%

%%%%%%%%%%%%%%%%%%%%%%%%%%%%%%%%%%%%%%%%%%%%%%%%%%%%%%%%%%%%%%%%%%%%%%%%%%%%%%
%%%%%%%%%%%%%%%%%%%%%%%%%%%%%%%%%%%%%%%%%%%%%%%%%%%%%%%%%%%%%%%%%%%%%%%%%%%%%%
%%
Let $\widetilde{\bf t}_{   {\bf p}_n } = \left(\rho_j({\bf t}_{   {\bf p}_j  }) \right)_n$ be another coordinate.
%%
%%%%%%%%%%%%%%%%%%%%%%%%%%%%%%%%%%%%%%%%%%%%%%%%%%%%%%%%%%%%%
%%
%%
 Then, using the above arguments, we construct analogously
a section $\overline{ \widetilde\F}( {\bf p}_n )$ by the formula
%%
%%%%%%%%%%%%%%%%%%%%%%%%%%%%%%%%%%%%%%%%%%%%%%%%%%%%%%%%%%%%%%%%%%%%%%%%%%%%%%%%%%%%%%%%%%
%%
\[
\overline  {\widetilde{\F} } ({\bf p}_n)= \left[ \left(   \widetilde{\bf y}_n, 
 \widetilde \F_{ i_{ \widetilde{\bf t}_{{\bf p}_n} } } ( \widetilde{\bf h}  )
   \right) \right]   
= \left[ \left( \widetilde \theta,  f({\widetilde {\bf x}}_n)  \right) \right]= \overline{F}( {\widetilde {\bf x}}_n), 
\]
for $\widetilde{\bf y}_n=(\widetilde \theta, \widetilde{\bf t}_{ {\bf p}_n })$. 
%%
%%%%%%%%%%%%%%%%%%%%%%%%%%%%%%%%%%%%%%%%%%%%%%%%%%%%%%%%%%%%%%%%%%%%%%%%%%%%%%%%%%%%%%%%%%%%%%%%%%
%%
 Since $\left(i^{-1}_{ \widetilde{\bf t}_{ {\bf p}_n } } \circ i_{{\bf t}_{ {\bf p}_n } }\right)$  
is an automorphism of $G_{{\bf p}_n}$, we represent a 
   change of variables $\widetilde{t}_{ p_j }=\rho_j(z_j)$ 
in terms of composition of trivializations 
\begin{equation}
\label{kozel}
\rho_j (z_j) \mapsto  i^{-1}_{ \widetilde{t}_{j, p}  } \circ i_{ t_{j, p} }, 
\end{equation}
and, therefore, 
relate $\F_{i_{\widetilde{\bf t}_{{\bf p}_n}}}(\widetilde {\bf h}_n)$ with $\F_{i_{{\bf t}_{{\bf p}_n}}}({\bf h}_n)$. 
%% 
%%%%%%%%%%%%%%%%%%%%%%%%%%%%%%%%%%%%%%%%%%%%%%%%%%%%%%%%%%%%%%%%%%%%%%%%%%%%%%%%%%%
%%
Since \eqref{kozel}  
defines a representation on $G$ of the group ${\rm Aut}_j\; \Oo^{(1)}_j$ of changes of coordinates, then  
 then $G_{ {\bf p}_n }$ is canonically identified with the twist of $G$ by the 
 ${\bf Aut}_n \; {\bm \Oo}^{(1)}_n$-torsor of formal coordinate at $p_j$.  
%%
%%%%%%%%%%%%%%%%%%%%%%%%%%%%%%%%%%%%%%%%%%%%%%%%%%%%%%%%%%%%%%%%%%%%%%%%%%%%%%%%%%
%%%%
%%%%%%%%%%%%%%%%%%%%%%%%%%%%%%%%%%%%%%%%%%%%%%%%%%%%%%%%%%%%%%%%%%%%%%%%%%%%%%%%%%%%%%%%%%%%%%%%%%%%%%%%%%%%%%%%%%
%%
Using definition of a torsor one sees that  
prescribed rational functions of the space $\Theta\left(n, k, G_{ {\bf t}_{{\bf p}_n} }, U\right)$
can be treated as 
${\bf Aut}_n\; {\bm \Oo}^{(1)}_n$-torsor  
of the product of groups of a coordinate transformation, namely,    
%%%%
%%
%%%%%%%%%%%%%%%%%%%%%%%%%%%%%%%%%%%%%%%%%%%%%%%%%%%%%%%%%%%%%%%%%%%%
%%
% as we already explained in the proof of Proposition \eqref{},  %6.4.7 
%%
that 
%%%%%%%%%%%%%%%%%%%%%%%%%%%%%%%%%%%%%%%%%%%%%%%%%%%%%%%%%%%%%%%%
%%
${\bf x}_n = \left( R(\rho_n)^{-1}. ({\bf g}_n),  \widetilde{\bf t}_{ {\bf p}_n } \right)$, 
 ${\bf y}_n = \left(\theta.R(\rho_n),  \widetilde{\bf t}_{ {\bf p}_n }\right)$.
Thus, we relate the l.h.s and r.h.s. of \eqref{mainfo}. 
%%

%%%%%%%%%%%%%%%%%%%%%%%%%%%%%%%%%%%%%%%%%%%%%%%%%%%%%%%%%%%%%%%%%%%%%%%%%%%%%%%%%%%%%%%%%%%%
To finish proof of proposition, it remains to show that $\overline{F}( {\bf x}_n)$ is invariant 
with respect to changes of coordinates.  We have the following 
%%
%%
%%%%%%%%%%%%%%%%%%%%%%%%%%%%%%%%%%%%%%%%%%%%%%%%%%%%%%%%5
%%%%%%%%%%%%%%%%%%%%%%%%%%%%%%%%%%%%%%%%%%%%%%%%%%%%%%%%%
\begin{lemma} 
\label{nezc}
For generic elements of the space of prescribed rational functions 
$F \left({\bf x}_n \right) \in \Theta\left(n, k, G_{ {\bf z}_n }, U\right)$  
for an admissible Lie algebra,  
 $\overline{F}({\bf x}_n)$   
   are canonical, i.e., independent on 
 changes 
\begin{eqnarray}
\label{zwrho}
{\bf z}_{n +k} \mapsto \widetilde{\bf z}_{n+k}= \left({\bm \rho}_i({\bf z}_i)\right)_{n+k}, \quad 1 \le i \le n+k, 
\end{eqnarray}
 as local coordinates of  ${\bf z}_n$ and $\widetilde{\bf z}_k$,  
 at points ${\bf p}_n$ and $\widetilde{\bf p}_k$. 
\end{lemma}
%%%%%%%%%%%%%%%%%%%%%%%%%%%%%%%%%%%%%%%%%%%%%%%%%%%%%%%%%%%%%%%%%%%%%%%%%%%%%%%%%%%%%%%%%%%%%%%
%%%%%%%%%%%%%%%%%%%%%%%%%%%%%%%%%%%%%%%%%%%%%%%%%%%%%%%%%%%%%%%%%%%%%%%%%%%%%%%%%%%%%%%%%%%%%%%
%%
\begin{remark}
%%
%%%%%%%%%%%%%%%%%%%%%%%%%%%%%%%%%%%%%%%%%%%%%%%%%%%%%%%%%%%%%%%%%%%%%%%%%%%5
%%  
 A generalization Lemma \ref{nezc}
 for the case of a arbitrary smooth manifold  
will be given in \cite{Z}. 
\end{remark}
%%
%%%%%%%%%%%%%%%%%%%%%%%%%%%%%%%%%%%%%%%%%%%%%%%%%%%%%%%%%%%%%%%%%%%%%%%%%%%%%%%%%%%%%%%%%%%%%%%%%%%
%%%%%%%%%%%%%%%%%%%%%%%%%%%%%%%%%%%%%%%%%%%%%%%%%%%%%%%%%%%%%%%%%%%%%%%%%%%%%%%%%%%%%%%%%%%%%%%%%%%
%\begin{proof}
%%
%%
%%%%%%%%%%%%%%%%%%%%%%%%%%%%%%%%%%%%%%%%%%%%%%%%%%%%%%%%%%%%%%%%%%%%%%%%%%%%%%%%%%%%%%%%%%%%%%%%
Indeed, consider the vector 
\begin{equation}
\label{overphi}
\overline{F}(\widetilde{\bf x}_n)= \left[ 
F \left( {\bf g}_n, \widetilde{\bf z}_n \; {\bf d} \widetilde {\bf z}_{ {\it i}(n)} \right)  \right].    
\end{equation}
%%
%%%%%%%%%%%%%%%%%%%%%%%%%%%%%%%%%%%%%%%%%%%%%%%%%%%%%%%%%%%%%%%%%%%%%%%%%%%%%%%%%%%%%%%%%%
Note that 
 $d\widetilde{z}_j = \sum\limits_{i=1}^n dz_i \;  
{\partial_{z_i} \rho_j}$,
%%  
%%
%%\nn
%%
$\partial_{z_i} \rho_j = \frac
{\partial \rho_j} {\partial z_i } $.    
By the definition of the action of ${\bf Aut}_n\; {\bm \Oo}^{(1)}_n$,  
when rewriting $d\widetilde{\bf z}_i$,
we have  
%%
%
%%
%%%%%%%%%%%%%%%%%%%%%%%%%%%%%%%%%%%%%%%%%%%%%%%%%%%%%%%%%%%%%%%%%%%%%%%%%%%%%%%%%%%%%%%%%%%%%%%%%%%
%%
%%  
\begin{eqnarray*}
&&   \overline{F}(\widetilde{\bf x}_n) =\overline{\F} ( 
{\bf g}_n, \widetilde{\bf z}_n \; {\bf d} \widetilde{ \bf z}_n)   
\\
&&
\qquad =  {\rm R}( {\bm \rho}_n) \; 
 \left[  F \left( {\bf g}_n, 
{\bf z}_n \; {\bf d} \widetilde{\bf z}_{i(n)} \right)
\right]
\\
 && \qquad =  {\rm R}({\bm \rho}_n) \; 
\left[  F \left( {\bf g}_n, 
{\bf z}_n \;  
\sum\limits_{j=1}^n \partial_j \rho_{i(n)} \; dz_j  \right)
\right]. 
\end{eqnarray*}
 By using \eqref{ldir1} and linearity of the mapping $F$, 
 we obtain from the last equation 
\begin{equation}
\label{norma}
\overline{F}(\widetilde{\bf x}_n) 
=
\overline{F} ({\bf g}_n, \widetilde{\bf z}_n \;{\bf d}\widetilde{z}_n)  
=
 \left[
 F \left(  {\bf g}_n, {\bf z}_n \; {\bf dz}_{i(n)} \right)   
\right], 
\end{equation}
with  
\begin{equation}
\label{rovno}
{\rm R} ({\bm \rho}_n)=  
\left[ \widehat \partial_{J} \rho_{i(I)}\right]
=
\left[
\begin{array}{c}
\widehat \partial_{J} \rho_{i_1(I)}
\\
\widehat \partial_{J } \rho_{i_2(I)} 
\\
\cdots  
\\
\widehat \partial_{J} \rho_{i_n(I)}    
\end{array}
\right]. 
\end{equation}
The index operator $J$ takes the value of index $z_j$ of arguments in the vector \eqref{norma},  
while the index operator $I$ takes values of index of differentials $dz_i$ in each entry of 
the vector $\overline{F}$ \eqref{overphi}. 
 Thus, the index operator 
$i(I)=(i_{I}, \ldots, i_n(I))$      
is given by consequent cycling permutations of $I$. 
Taking into account the property \eqref{ldir1},  
we define the operator 
\begin{equation}
\label{hatrho}
\widehat\partial_{J } \rho_a =   
\exp\left(   - \sum\limits_{ {\bf r}_n, \; 
\sum\limits_{i=1}^n r_i  \ge 1 }  
r_J\; \beta^{(a)}_{{ \bf r}_n }\; \zeta^{r_1}_1 \;  \ldots \; 
 \zeta^{r_J}_J \ldots \zeta^{r_n}_n \; \partial_{z_J} 
 \right), 
\end{equation}
which contain index operators $J$ as index of a  
dummy variable $\zeta_J$ turning into $z_j$, $j=1, \ldots, n$.  
\eqref{hatrho} acts on each argument of maps $F$ in the vector $\overline{F}$ \eqref{overphi}.   
Due to properties of the Lie algebra $\mathcal G$ required above, 
the action of operators $R\left({\bm \rho}_n\right)$ 
 on ${\bf g}_n \in G$ results in a sum of finitely many terms.
%%
%%%%%%%%%%%%%%%%%%%%%%%%%%%%%%%%%%%%%%%%%%%%%%%%%%%%%%%%%%%%%%%%%%%%%%%%%
  In  \cite{BZF}, it is proven 
%%%%%%%%%%%%%%%%%%%%%%%%%%%%%%%%%%%%%%%%%%%%%%%%%%%%%%%%%%%%%%%%%%%%%%%%%
%% 
%%
%%%%%%%%%%%%%%%%%%%%%%%%%%%%%%%%%%%%%%%%%%%%%%%%%%%%%%%%%%%%%%%%%%%%%%%%%%%%%%%%%%
%%
\begin{lemma}
\label{podnik}
The mappings  
\[
{\bm \rho}_n({\bf z}_j ) \mapsto R 
\left({\bm \rho}_n \right), 
\]
for $1 \le j \le n$, 
 define a representation of  ${\bf Aut}_n \; {\bm \Oo}_n$
on ${\bf G}_{{\bf z}_n}$ by  
\[
{\rm R} 
\left(\rho \circ \widetilde{\rho}\right) = {\rm R} 
\left(\rho\right) \; {\rm R} 
\left(\widetilde{\rho}\right), 
\]
for $\rho$, $\widetilde{\rho} \in {\bf Aut}_n \; {\bm \Oo}_n$.  
\end{lemma}  
%%
%%%%%%%%%%%%%%%%%%%%%%%%%%%%%%%%%%%%%%%%%%%%%%%%%%%%%%%%%%%%%
%%
Using Lemma \ref{podnik},  
 we then conclude that  
 the vector $\overline{F}$ \eqref{overphi} is invariant, 
i.e., 
\begin{equation*}
\label{phiproperty}
 \overline{F} (\widetilde{\bf x}_n) = \overline{F} \left(
{\bf g}_n, \widetilde{\bf z}_n\; {\bf d}\widetilde{\bf z}_n \right) 
=  \overline{F} \left( {\bf g}_n, {\bf z}_n \;{\bf dz}_n\right)= \overline{F} ({\bf x}_n).     
\end{equation*}
%%
%%%%%%%%%%%%%%%%%%%%%%%%%%%%%%%%%%%%%%%%%%%%%%%%%%%%%%%%%%%%%%%%%%%%%%%%%%%%%%%%%%
%%%%%%%%%%%%%%%%%%%%%%%%%%%%%%%%%%%%%%%%%%%%%%%%%%%%%%%%%%%%%%%%%%%%%%%%%%%%%%%%%%
%%
The definition of prescribed rational functions 
$F({\bf x}_n) \in \Theta\left(n, k, G_{ {\bf z}_n }, U\right)$ 
consists of two conditions on $F$. 
 The first requires the existence of positive 
integers $\beta^n_m(v_i, v_j)$ depending on $v_i$, $v_j$ only, and the second 
restricts orders of poles of corresponding sums. 
%%%%
%%%%%%%%%%%%%%%%%%%%%%%%%%%%%%%%%%%%%%%%%%%%%%%%%%%%%%%%%%%%%%%%%%%%%%%%%%%%%%%%%%%%%%%%%%%%%%%%%%%
%%%%%%%%%%%%%%%%%%%%%%%%%%%%%%%%%%%%%%%%%%%%%%%%%%%%%%%%%%%%%%%%%%%%%%%%%%%%%%%%%%%%%%%%%%%%%%%%%%%
%%
The insertions of Lie algebra $k$ elements $\left({\bf g}_{k}, {\bf t}_{ {\bf p}_k} \; {\bf dt}_{{\bf p}_k} \right)$  
 which are present in the definition  of prescribed rational functions 
keep functions $F$  invariant with respect to coordinate changes \eqref{zwrho}.   
Thus, the construction of spaces \ref{poyma} is invariant under the action of the group 
${\bf Aut}_{n+k} \; {\bm \Oo}^{(1)}_{n+k}$.  
\end{proof}
%%%%
%%%%%%%%%%%%%%%%%%%%%%%%%%%%%%%%%%%%%%%%%%%%%%%%%%%%%%%%%%%%%%%%%%%%%%%%%%%%%%%%%%%%%%%%%%%%
%%%%%%%%%%%%%%%%%%%%%%%%%%%%%%%%%%%%%%%%%%%%%%%%%%%%%%%%%%%%%%%%%%%%%%%%%%%%%%%%%%%%%%%%%%%%
\subsection{The bundle dual to $\W_M$}  
It is still possible to define a fiber bundle in the dual formulation 
 when the conditions on grading subspaces of $G_{{\bf p}_n}$ are missing. 
%%
%%%%%%%%%%%%%%%%%%%
%%
The advantage of the dual (defined with respect to an appropriate form) fiber bundle  
${\W}^\dagger_{M}$  
is that to define it we do not need to assume that the $\C$-grading  
on $G_{ {\bf p}_n }$ is bounded from below or that the graded components are finite-dimensional. 
Nevertheless, we have to assume that a Lie algebra $\mathcal G$ satisfies remaining conditions of
subsection \ref{algebra}. 
%%%%
%%%%%%%%%%%%%%%%%%%%%%%%%%%%%%%%%%%%%%%%%%%%%%%%%%%%%%%%%%%%%%%%%%%%%%%%%%%%%
%%%%%%%%%%%%%%%%%%%%%%%%%%%%%%%%%%%%%%%%%%%%%%%%%%%%%%%%%%%%%%%%%%%%%%%%%%%%%
Introduce the canonical residue map 
\begin{equation}
\label{resu}
{\rm Res}\;_{ {\bf t}_{{\bf p}_n} } : {\bf t}_{{\bf p}_n}  \longrightarrow \C^n,   
\end{equation}
with separate residues for each variable. 
%%%%
%%%%%%%%%%%%%%%%%%%%%%%%%%%%%%%%%%%%%%%%%%%%%%%%%%%%%%%%%%%%%%%%%%%%%%%%%%%
%%%
Since the space of differentials  
${\bm \Omega}_{{\bf p}_n}$ on the 
punctured discs ${\bf D}^\times_{{\bf p}_n}$ is already represented in the definition 
of $F({\bf p}_n)$,  
then the the map \eqref{resu} 
gives rise to a pairing
\begin{eqnarray*}
  \gamma \left({\W}^\dagger_M|_{{\bf D}^\times_{ {\bf p}_n }} \right)  \times 
  \gamma \left(\W_M|_{ {\bf D}^\times_{  {\bf p}_n } } \right)   \rightarrow \C^n,  
\nn
\eta, \mu \mapsto {\rm Res}\;_{{\bf t}_{ {\bf p}_n } } \left(  \eta, \mu \right), 
\end{eqnarray*} 
for $\eta \in \gamma \left( {\W}^\dagger|_{D^\times_{{\bf p}_n}}\right)$ 
and $\mu \in \gamma\left(\W|_{D^\times_{{\bf p}_n}}\right)$ 
belong to corresponding space of fibers.   
%%
%%
%%%%%%%%%%%%%%%%%%%%%%%%%%%%%%%%%%%%%%%%%%%%%%%%%%%%%%%%%%%%%%%%%%%%%%%%%%%%%%%%%%%%%%
%%%%%%%%%%%%%%%%%%%%%%%%%%%%%%%%%%%%%%%%%%%%%%%%%%%%%%%%%%%%%%%%%%%%%%%%%%%%%%%%%%%%%%
%%
Using this pairing, we obtain 
for each fiber $\mu$ of $\W_M|_{{\bf D}^\times_{{\bf p}_n}}$,    
a linear 
operator on $\W_M|_{{\bf D}^\times_{{\bf p}_n}}$ given by 
${\rm Res}\;_{{\bf t}_{ {\bf p}_n} } \left( {\W}^\dagger_M, \mu \right)$. 
%%
%%
%%%%%%%%%%%%%%%%%%%%%%%%%%%%%%%%%%%%%%%%%%%%%%%%%%%%%%%%%%%%%%%%%%%%%%%%%%%%%%%%%
%%
Thus, we obtain a well-defined linear map 
\begin{equation}
\label{tosca}
 {\W}^\dagger_{ {\bf D}\times_{ {\bf p}_n} }: \gamma\left(\W_M|_{{\bf D}\times_{ {\bf p}_n }}
\right)  
\rightarrow {\rm End} \left( \gamma \left(\W_M|_{ {\bf D}\times_{{\bf p}_n}
  }\right) \right). 
\end{equation}
%%
%%%%
%%
%%%%%%%%%%%%%%%%%%%%%%%%%%%%%%%%%%%%%%%%%%%%%%%%%%%%%%%%%%%%%%%%%%%%%%%%%%%%%%%%%%%%%%%%%%%%%%%%%%%%
%%
For formal coordinates ${\bf t}_{{\bf p}_n}$ on ${\bf D}^\times_{{\bf p}_n}$, 
 a fiber $\mu = f({\bf x})$  
of $W_M|_{{\bf D}^\times_{{\bf p}_n}}
$ 
with $g_{{\bf p}_n} \in G_{{\bf p}_n}$ 
 with respect to the ${\bf t}_{{\bf p}_n}$-trivialization, the map 
 \eqref{tosca} is just ${\bf x}_n$.  
The geometrical information contained in ${\W}^\dagger_M|_{ {\bf D}^\times_{{\bf p}_n}
  }$ 
is equivalent to that of  
$\W_M|_{ {\bf D}^\times_{{\bf p}_n}
 }$. 
%%
%%%%

%%%%%%%%%%%%%%%%%%%%%%%%%%%%%%%%%%%%%%%%%%%%%%%%%%%%%%%%%%%%%%%%%%%%%%%%%%%%%%%%%%%%%%%%
%%%%%%%%%%%%%%%%%%%%%%%%%%%%%%%%%%%%%%%%%%%%%%%%%%%%%%%%%%%%%%%%%%%%%%%%%%%%%%%%%%%%%%%%
%%
Proposition \ref{mainpro} 
provides us with a way how to,  
starting from an admissible infinite-dimensional Lie algebra 
construct explicitly 
 a fiber  
bundle $\W_M$ over a smooth complex curve $M$,
with canonical sections $\F({{\bf p}_n})$ of $\W_M|_{ {\bf D}_{ {\bf p}_n } }$ 
and fibers 
with values in ${\rm End} \left(G_{ {\bf p}_n }\right)$ for any collection of non-intersecting disks 
${\bf D}_{{\bf p}_n}$ on $M$. 
%%%%
%%
%%%%%%%%%%%%%%%%%%%%%%%%%%%%%%%%%%%%%%%%%%%%%%%%%%%%%%%%%%%%%%%%%%%%%%%%%%%%%%%%%%%%%%%%%%%
%%%%%%%%%%%%%%%%%%%%%%%%%%%%%%%%%%%%%%%%%%%%%%%%%%%%%%%%%%%%%%%%%%%%%%%%%%%%%%%%%%%%%%%%%%%
%%
%% 
%% 
Due to the assumptions in the definition of an admissible Lie algebra, the filtration    
%%
%%%%%%%%%%%%%%%%%%%%%%%%%%%%%%%%%%%%%%%%%%%%%%%%%%%%%%%%%%%%%%%%%%%%%%%%%%%%%%%%%%%%%%%%%
%%
  $G_{   {\bf t}_{ {\bf p}_n }   , \le m } = \bigoplus^m_{n=K} G_{   {\bf t}_{  {\bf p}_n  }, n    }$,    
is preserved by ${\bf Aut}_n \; {\bm \Oo}_n^{(1)}$.
Then it is possible to prove that 
 the exact sequences of ${\bf Aut}_n\; {\bm \Oo}^{(1)}_n$-modules 
\[
0 \rightarrow G_{     {\bf t}_{  {\bf p}_n  },   \le (m - 1)  }  
\rightarrow  G_{ {\bf t}_{ {\bf p}_n }, \le m } \rightarrow  G_{ {\bf t}_{  {\bf p}_n }, m }  \rightarrow  0,  
\]
gives rise to an exact sequence of vector bundles
\[
0 \rightarrow  \W_M^{  \le (m-1)  } \rightarrow \W_M^{ \le m } \rightarrow \W_M^m \rightarrow 0.
\]
%%%%
%%
%%%%%%%%%%%%%%%%%%%%%%%%%%%%%%%%%%%%%%%%%%%%%%%%%%%%%%%%%%%%%%%%%%%%%%%%%%%%%%%%%%%%%%%%%%%%%%
%%%%%%%%%%%%%%%%%%%%%%%%%%%%%%%%%%%%%%%%%%%%%%%%%%%%%%%%%%%%%%%%%%%%%%%%%%%%%%%%%%%%%%%%%%%%%%
\section{Applications}
\label{further}
%% 
%%
%%%%%%%%%%%%%%%%%%%%%%%%%%%%%%%%%%%%%%%%%%%%%%%%%%%%%%%%%%%%%%%%%%%%%%%%%%%%%
In this section we list multiple  
 applications of the notion of the bundle of prescribed
rational functions on complex manifolds  
\cite{Fei, Wag,  DiMaSe, TUY} 
in deformation theory \cite{Ma, BG, HinSch}, 
and algebraic topology of foliations \cite{Bott}.  
%%
%%%%
%%%%%%%%%%%%%%%%%%%%%%%%%%%%%%%%%%%%%%%%%%%%%%%%%%%%%%%%%%%%%%%%%%%%%%%%%%%%%%

%%%%%%%%%%%%%%%%%%%%%%%%%%%%%%%%%%%%%%%%%%%%%%%%%%%%%%%%%%%%%%%%%%%%%%%%%
%%%%%%%%%%%%%%%%%%%%%%%%%%%%%%%%%%%%%%%%%%%%%%%%%%%%%%%%%%%%%%%%%%%%%%%%% 
%%
%%
%%
The fiber bundle associated to spaces of prescribed 
rational functions on domains of arbitrary complex manifolds  
can be used in construction of generalizations of the Bott--Segal theorem \cite{BS}.  
%%  
%%
%%%%%%%%%%%%%%%%%%%%%%%%%%%%%%%%%%%%%%%%%%%%%%%%%%%%%%%%%%%%%%%%%%%%%%%%%%%%%%%%%%%%%%%%%%%
%%%%%%%%%%%%%%%%%%%%%%%%%%%%%%%%%%%%%%%%%%%%%%%%%%%%%%%%%%%%%%%%%%%%%%%%%%%%%%%%%%%%%%%%%%%
%%
As it was demonstrated in \cite{Wag}, the ordinary 
 cohomology of vector fields on complex manifolds turns to be not the most effective and general one.   
In order to avoid trivialization 
and reveal a richer cohomological structure  
of complex manifolds cohomology, one has to 
treat \cite{Fei}  holomorphic vector fields as a sheaf rather than taking 
global sections. 
%% 
%%%%%%%%%%%%%%%%%%%%%%%%%%%%%%%%%%%%%%%%%%%%%%%%%%%%%%%%%%%%%%%%%%%%%%%%%%%%%%%%%%%%%%
%%
In analogy with the construction of \cite{BS}, the cohomology of a foliation  
over a smooth complex curve $M$ can be expressed in terms of cohomology 
of a canonical complex $\left(C^n_M, \delta^n_m \right)$ for an auxillarly bundle $\W_M$ of prescribed 
rational functions 
 with intrinsic action of $\delta^n_m$. 
Constructions presented in this paper are also useful for purposes of cosimplitial cohomology \cite{Wag}
and computations in 
the 
 deformation theory of complex manifolds 
\cite{Ma, Fei, HinSch, GerSch, Kod}. 
%%

%%%%%%%%%%%%%%%%%%%%%%%%%%%%%%%%%%%%%%%%%%%%%%%%%%%%%%%%%%%%%%%%%%%%%%%%%%%%%%%
%%%%%%%%%%%%%%%%%%%%%%%%%%%%%%%%%%%%%%%%%%%%%%%%%%%%%%%%%%%%%%%%%%%%%%%%%%%%%%%
%%
%% 
Prescribed rational functions approach is applicable to studies of cohomology 
and characteristic classes of 
foliations of complex manifolds. 
%%
%%%%%%%%%%%%%%%%%%%
%%
In \cite{Lo} a smooth structure on the leaf space
$M/\mathfrak F$ of a foliation $\mathfrak F$ of codimension $k$ on
a smooth manifold $M$ that allows to apply to $M/\mathfrak F$ the
same techniques as to smooth manifolds.
%%
%%%%%%%%
%% 
In \cite{Lo} characteristic classes for a
foliation as  elements of the cohomology of certain bundles over the
leaf space $M/\mathfrak F$ are defined.   
%%
%%%%%%%
It would be interesting to develop also intrinsic 
and purely coordinate independent theory of a smooth manifold 
foliation cohomology involving bundles of prescribed rational functions.  
%%
%%%%%%%%%%%%%%%%%%%%%%%%%%%%%%%%%%%%%%%%%%%%%%%%%%%%%%%%%%%%%%%%%%%%%%%%%%%%%%%%%%%%%%%%%%%%%%%
%%%%%%%%%%%%%%%%%%%%%%%%%%%%%%%%%%%%%%%%%%%%%%%%%%%%%%%%%%%%%%%%%%%%%%%%%%%%%%%%%%%%%%%%%%%%%%%


\begin{thebibliography}{99}


%% 
\bibitem[BG]{BG}  
Braverman, A., Gaitsgory, D. 
%%
Deformations of local systems and Eisenstein series. Geom. Funct. Anal. 17 (2008), no. 6, 1788--1850.  
 
\bibitem[BR]{BR}  J.N. Bernstein and B.I. Rosenfeld, 
Homogeneous spaces of infinite-dimensional Lie algebras and the characteristic classes of foliations, 
Russ. Math. Surv. 28 (1973) no. 4,
107--142.

%%
\bibitem[BD]{BD} A. Beilinson and V. Drinfeld, Chiral algebras, Preprint. 

%%
\bibitem[BZF]
{BZF} 
 Frenkel, E.; Ben-Zvi, D. Vertex algebras and algebraic curves. Mathematical Surveys and Monographs,
 88. American Mathematical Society, Providence, RI, 2001. xii+348 pp. 


\bibitem[Bott]
{Bott} R.\,Bott, Lectures on characteristic classes and foliations.
 Springer LNM 279 (1972), 1--94.

%%
\bibitem[BS]{BS} R. Bott, G.Segal, The cohomology of the vector fields on a manifold, Topology
Volume 16, Issue 4, 1977, Pages 285--298. 


\bibitem[DiMaSe]{DiMaSe}%
 Di Francesco, P., Mathieu, P., S\'en\'echal, D.:
Conformal Field Theory. Graduate Texts in Contemporary Physics,
Springer 1996


\bibitem[Fuks]{Fuks} Fuks, D. B.:Cohomology of Infinite Dimensional Lie
algebras. New York and London: Consultant Bureau 1986


\bibitem[Fei]{Fei} Feigin, B. L.: Conformal field theory and
Cohomologies of the Lie algebra of holomorphic vector fields on a
complex curve. Proc. ICM, Kyoto, Japan, 71-85 (1990) 


\bibitem[GerSch]{GerSch} Gerstenhaber, M., Schack, S. D.: Algebraic Cohomology
and Deformation Theory, in: Deformation Theory of Algebras and
Structures and Applications, NATO Adv. Sci. Inst. Ser. C {\bf 247},
Kluwer Dodrecht 11-264 (1988)

%%%%%%%%%%%%%%%%%%%%%%%%%%%%%%%%%%%%%%%%%%%%%%%%%%%%%%%%%%%%%%%%%%%%%%%%%%%%

\bibitem[GK]{GK} I.M. Gelfand and D.A. Kazhdan,
 Some problems of differential geometry and the calculation of cohomologies of Lie algebras 
of vector fields, Soviet Math Doklady 12 (1971)
1367--1370.

\bibitem[GKF]{GKF} I.M. Gelfand, D.A. Kazhdan and D.B. Fuchs, The actions of infinite--dimensional Lie
algebras, Funct. Anal. Appl. 6 (1972) 9--13.

%%%%%%%%%%%%%%%%
\bibitem [HinSch]
%%
 {HinSch} Hinich V.,  Schechtman, V.: Deformation Theory and Lie
algebra Homology. 
%%
I. Algebra Colloq. 4 (1997), no. 2, 213--240; II no. 3, 291--316. 

%%%%%%%%%%%%%%%%%%%%%%%%%%%%%%%%%%%%%%%%%%%%%%%5%%
\bibitem[H1]
{H1} { Huang Y.-Zh.} 
A cohomology theory of grading-restricted vertex algebras. 
Comm. Math. Phys. 327 (2014), no. 1, 279-307. 

\bibitem[H2]
{H2}
Y.-Z. Huang, {\em Two-dimensional conformal geometry and vertex operator
algebras}, Progress in Mathematics, Vol. 148,
Birkh\"{a}user, Boston, 1997.

%%%%%%%%%%%%%%%%%%%%%%%%%%%%%%%%%%%%%%%%%%%%%%%%%%%%%%%%%%%%%%%%%%%%%%%%
%%

\bibitem[Kod]{Kod} Kodaira, K.: Complex Manifolds and Deformation
of Complex Structures. Springer Grundlehren {\bf 283} Berlin Heidelberg New
York 1986

\bibitem[Lo]{Lo} M.  V. Losik
%%
Orbit spaces and leaf spaces of foliations as generalized manifolds, arXiv:1501.04993.  

%%%%%%%%%%%%%%%%%%%%%%%%%%%%%%%%%%%%%%%%%%%%%%%%%%%%%%%%%%%%%%

\bibitem[Ma]{Ma} 
%%
Manetti M. Lectures on deformations of complex manifolds  
(deformations from differential graded viewpoint). Rend. Mat. Appl. (7) 24 (2004), no. 1, 1--183.

\bibitem[PT]{PT}
%%
Patras, F., Thomas, J.--C. Cochain algebras of mapping spaces and finite group actions. 
%%
Topology Appl. 128 (2003), no. 2--3, 189--207.
%%

%%%%%%%%%%%%%%%%%%%%%%%%%%%%%%%%%%%%%%%%%%%%%%%%%%%%%%%%%%%%%%%%%%%%%%%%%%%%
\bibitem[Sch]{Sch} Schiffer, M.: 
Half-order differentials on Riemann surfaces, 
SIAM J.Appl.Math.  \textbf{14}  (1966) 922--934. 

\bibitem[S]{S} 
%%
Sheinman, O. K. Global current algebras and localization on Riemann surfaces. 
%%
Mosc. Math. J. 15 (2015), no. 4, 833--846. 

\bibitem[Sm]{Sm}
%%
Smith, S. B. The homotopy theory of function spaces: a survey. 
Homotopy theory of function spaces and related topics, 3--39, 
Contemp. Math., 519, Amer. Math. Soc., Providence, RI, 2010. 


\bibitem[TUY]
{TUY}
A. Tsuchiya, K. Ueno and Y. Yamada, Conformal field theory on universal 
family of stable curves with gauge symmetries, in: {\em Advanced Studies in
Pure Math.}, Vol. 19, Kinokuniya Company Ltd.,
Tokyo, 1989, 459--566.


\bibitem[Wag]{Wag} Wagemann, F.: Differential graded cohomology and Lie algebras of holomorphic vector fields. 
Comm. Math. Phys. 208 (1999), no. 2, 521--540

%%
  \bibitem[Wi]{Wi} E. Witten, Quantum field theory, Grassmannians and algebraic curves, Comm. Math.
Phys. 113 (1988) 529--600.
%%


\bibitem[Z] 
{Z} A. Zuevsky. Prescribed rational functions cohomology of foliations on smooth manifolds. To appear. 

%%%%%%%%%%%%%%%%%%%%%%%%%%%%%%%%%%%%%%%%%%%%%%%%%%%%%%%%%%%%%%%%%%%%%%%%%%%%%

\end{thebibliography}
\end{document}